\documentclass[envcountsame,envcountsect]{svmult}
\usepackage[T1]{fontenc}

\usepackage{color}

\usepackage{makeidx,graphicx,multicol,footmisc}

\usepackage[leqno]{amsmath}
\usepackage{amsfonts}
\usepackage{amssymb}

\newcommand{\thmcapfont}{\bfseries\upshape}
\newcommand{\thmlabelfont}{\itshape}
\spnewtheorem{assumption}[theorem]{Assumption}{\bfseries}{}
\spnewtheorem*{thm-non}{Theorem }{\thmcapfont}{\thmlabelfont}
\spnewtheorem*{prop-non}{Proposition}{\thmcapfont}{\thmlabelfont}
\spnewtheorem*{classicalTorelli}{Classical Torelli Theorem for K3 surfaces}{\thmcapfont}{\thmlabelfont}
\spnewtheorem*{derivedTorelli}{Derived Torelli Theorem for K3 surfaces }{\thmcapfont}{\thmlabelfont}
\spnewtheorem*{WeakPolarisedTorelli}{Weak Torelli Theorem for polarised K3 surfaces}{\thmcapfont}{\thmlabelfont}
\spnewtheorem*{StrongPolarisedTorelli}{Strong Torelli Theorem for polarised K3 surfaces}{\thmcapfont}{\thmlabelfont}
\spnewtheorem*{SemiPolarisedTorelli}{Torelli Theorem for semi-polarised K3 surfaces}{\thmcapfont}{\thmlabelfont}
\spnewtheorem*{definition-non}{Definition}{\bfseries}{}
\spnewtheorem*{conj-non}{Conjecture }{\bfseries}{}
\spnewtheorem*{remarks-non}{Remarks}{\bfseries}{}
\smartqed

\usepackage{mdwlist}
\desclabelwidth{0em}

\usepackage[all]{xy}
\SelectTips{cm}{12}   

\newcommand{\bib}[5]{{\bibitem{#1}#2, {\emph{#3},} #4#5.}}

\newcommand{\sqmat}[4]{{\big(\genfrac{.}{.}{0pt}{1}{#1}{#3} \,
                        \genfrac{.}{.}{0pt}{1}{#2}{#4}\big) }}

\newcommand{\orth}{^\perp}
\newcommand{\inv}{^{-1}}

\newcommand{\Coh}{\mathrm{Coh}}
\newcommand{\Db}{\mathrm{Db}}
\newcommand{\Hom}{\mathrm{Hom}}
\newcommand{\image}{\mathrm{im}}
\newcommand{\Aut}{\mathrm{Aut}}
\newcommand{\IZ}{\mathbb{Z}}
\newcommand{\IQ}{\mathbb{Q}}
\newcommand{\IR}{\mathbb{R}}
\newcommand{\IC}{\mathbb{C}}
\newcommand{\IN}{\mathbb{N}}
\newcommand{\IP}{\mathbb{P}}
\newcommand{\kg}{\mathcal{G}}
\newcommand{\kf}{\mathcal{F}}
\newcommand{\ko}{\mathcal{O}}

\newcommand{\kc}{\mathcal{C}}
\newcommand{\ke}{\mathcal{E}}
\newcommand{\kp}{\mathcal{P}}
\newcommand{\kq}{\mathcal{Q}}

\newcommand{\LL}{\mathsf{L}}
\newcommand{\RR}{\mathsf{R}}
\newcommand{\FM}{\mathsf{FM}}

\DeclareMathOperator{\OO}{O}

\newcommand{\id}{\mathrm{id}}

\newcommand{\maps}[1]{\stackrel{#1}{\to}}
\newcommand{\antimaps}[1]{\stackrel{#1}{\leftarrow}}
\newcommand{\embed}{\hookrightarrow}
\newcommand{\isom}{ \text{{\hspace{0.48em}\raisebox{0.8ex}{${\scriptscriptstyle\sim}$}}}
                    \hspace{-0.55em}{\rightarrow}\hspace{0.3em}} 
\newcommand{\surject}{\rightarrow\hspace{-1.8ex}\rightarrow}
\newcommand{\bijection}{ \stackrel{1:1}{=} }

\renewcommand{\div}{\text{div}}
\newcommand{\sign}{\text{sgn}}
\newcommand{\disc}{\text{disc}}
\newcommand{\rk}{\text{rk}}
\newcommand{\eq}{\text{eq}}
\newcommand{\hodge}{{\mathsf{H}}}
\newcommand{\kd}{{\mathrm{K3}}}
\newcommand{\zd}{{2d}}

\newcommand{\pp}[1]{(\cdot,\cdot)} 

\newcommand{\Kthree}{{\mathop{\mathrm {K3}}\nolimits}}

\newcommand{\mapH}{\mathsf{H}}
\newcommand{\mapL}{\mathsf{L}}
\begin{document}

\title*{Fourier-Mukai partners and polarised $\Kthree$ surfaces}
\author{K. Hulek \and D. Ploog}
\institute{K. Hulek \at Institut f\"ur Algebraische Geometrie, Welfengarten 1, 30167 Hannover
           \email{hulek@math.uni-hannover.de}
      \and D. Ploog \at Institut f\"ur Algebraische Geometrie, Welfengarten 1, 30167 Hannover
           \email{ploog@math.uni-hannover.de}}
\date{}

\maketitle

\abstract{The purpose of this note is twofold. We first review the theory of 
Fourier-Mukai partners together with the relevant part of Nikulin's theory of
lattice embeddings via discriminants. Then we consider Fourier-Mukai partners
of $\Kthree$ surfaces in the presence of polarisations, in which case we prove
a counting formula for the number of partners.
%
%
}

\bigskip 
\setcounter{minitocdepth}{2}
\dominitoc

\noindent
The theory of FM partners has played a crucial role in algebraic geometry and its connections to string
theory in the last 25 years. Here we shall concentrate on a particularly interesting aspect 
of this, namely the theory of FM partners of $\Kthree$ surfaces.
We shall survey some of the most significant results in this direction. Another aspect, and this
has been discussed much less in the literature, is the question of Fourier-Mukai partners in the presence 
of polarisations. We shall also investigate this in some detail, and it is here that the  
paper contains some new results.

To begin with we review in Section ~\ref{sec:review} the use of derived
categories in algebraic geometry focusing on Fourier-Mukai partners.
In Sections~\ref{sec:lattices} and \ref{sec:overlattices} we then give a
self-contained introduction to lattices and lattice embeddings with emphasis
on indefinite, even lattices. This contains a careful presentation of Nikulin's
theory as well as some enhancements which will then become important for our counting formula. 
From Section~\ref{sec:K3} onwards we will fully concentrate
on $\Kthree$ surfaces. After recalling the classical as well as Orlov's derived Torelli theorem 
for $\Kthree$ surfaces we describe the counting formula for the number of 
$\Kthree$ surfaces given by Hosono, Lian, Oguiso, Yau \cite{HLOY2002}. 
In Section \ref{sec:polarisedK3surfaces} we discuss polarised $\Kthree$ surfaces and their moduli.
The relationship between polarised $\Kthree$ surfaces and FM partners was 
discussed by Stellari in
\cite{stellari} and \cite{stellari-group}. Our main result in this direction is a counting formula
given in Section \ref{sec:countingformula}
in the spirit of \cite{HLOY2002}.

In a number of examples we will discuss the various phenomena which
occur when considering Fourier-Mukai partners in the presence of polarisations.

\bigskip\noindent
\textbf{Conventions:} We will denote bijections of sets as $A\bijection B$. Also, all group
actions will be left actions. In particular, we will denote the sets of orbits by $G\backslash A$
whenever $G$ acts on $A$. However, factor groups are written $G/H$.

If we have group actions by $G$ and $G'$ on a set $A$ which are compatible (i.e.\ they commute),
then we consider this as a $G\times G'$-action (and not as a left-right bi-action). In particular,
the total orbit set will be written as $G\times G'\backslash A$ (and not $G\backslash A/G'$).

\begin{acknowledgement}
\ 
We are grateful to M.~Sch\"utt and F.~Schulz for discussions concerning lattice theory.
The first author would like to thank the organizers of the Fields Institute Workshop on
Arithmetic and Geometry of $\Kthree$ surfaces and Calabi-Yau threefolds held in August 2011
for a very interesting and stimulating meeting. The second author has been supported by 
DFG grant Hu 337/6-1 and by the DFG priority program 1388 `representation theory'.
\end{acknowledgement}


\section{Review Fourier-Mukai partners of $\Kthree$ surfaces}\label{sec:review}

\noindent
For more than a century algebraic geometers have looked at the classifiation of varieties up to birational
equivalence. This is a weaker notion than biregular isomorphism which, however, captures
a number of crucial and interesting properties.

About two decades ago, a different weakening of biregularity has emerged in algebraic
geometry: derived equivalence. Roughly speaking, its popularity stems from two reasons:
on the one hand, the seemingly ever increasing power of homological methods in all
areas of ma\-the\-ma\-tics, and on the other hand the intriguing link derived categories
provide to other mathematical disciplines such as symplectic geometry and representation
theory as well as to theoretical physics.

\subsection*{History: derived categories in algebraic geometry}

Derived categories of abelian categories were introduced in the 1967 thesis of
Grothendieck's student Verdier \cite{verdier}.
The goal was to set up the necessary
homological tools for defining duality in greatest generality --- which meant
getting the right adjoint of the push-forward functor $f_*$. This adjoint cannot
exist in the abelian category of coherent sheaves; if it did, $f_*$ would be exact.
Verdier's insight was to embed the abelian category into a bigger category with
desirable properties, the derived category of complexes of coherent sheaves. The reader is
referred to \cite{hartshorne-duality} for an account of this theory.

In this review, we will assume that the reader is familiar with the basic theory
of derived categories \cite{gelfand-manin}, \cite{weibel}. An exposition of the theory of
derived categories in algebraic geometry can be found in
two text books, namely  by Huybrechts \cite{huybrechts} and
by Bartocci, Bruzzo, Hern\`andez-Ruip\'erez \cite{BBHR}. We will denote by $D^b(X)$
the bounded derived category of coherent sheaves. This category is particularly well
behaved if $X$ is a smooth, projective variety. Later on we will consider $\Kthree$ surfaces,
but in this section, we review some general results.

We recall that two varieties $X$ and $Y$ are said to be \emph{derived equivalent}
(sometimes shortened to \emph{D-equivalent}) if there is an exact equivalence of
categories $D^b(X)\cong D^b(Y)$.

It should be mentioned right away that the use of the \emph{derived} categories is
crucial: a variety is uniquely determined by the abelian category of coherent sheaves,
due to a theorem of Gabriel \cite{gabriel}. Thus, the analogous definition using
abelian categories does not give rise to a new equivalence relation among varieties.

After their introduction, derived categories stayed in a niche, mainly considered as
a homological bookkeeping tool. Its uses were to combine the classical derived
functors into a single derived functor, or to put the Grothendieck spectral sequence
into a more conceptual framework.
The geometric use of derived categories started with the following groundbreaking result:

\begin{thm-non}[Mukai, 1981 \cite{mukai-abelian}]
Let $A$ be an abelian variety with dual abelian variety $\hat A$. Then $A$ and
$\hat A$ are derived equivalent.
\end{thm-non}

Since an abelian variety and its dual are in general not isomorphic (unless they are
principally polarised) and otherwise
never birationally equivalent, this indicates a new phenomenon. For the proof,
Mukai employs the Poincar\'e bundle $\kp$ on $A\times\hat A$ and investigates the
functor $D^b(A)\to D^b(\hat A)$ mapping $E\mapsto\RR\hat\pi_*(\kp\otimes\pi^*E)$
where $\hat\pi$ and $\pi$ denote the projections from $A\times\hat A$ to $\hat A$
and $A$ respectively.

Mukai's approach was not pursued for a while. Instead, derived categories were
used in different ways for geometric purposes: Beilinson, Bernstein, Deligne \cite{BBD}
introduced perverse sheaves as certain objects in the derived category of constructible
sheaves of a variety in order to study topological questions. The school around
Rudakov introduced exceptional collections (of objects in the derived category), which
under certain circumstances leads to an equivalence of $D^b(X)$ with the derived
category of a finite-dimensional algebra \cite{rudakov}. It should be mentioned that
around the same time, Happel introduced the use of triangulated categories in
representation theory \cite{happel}.

\subsection*{Derived categories as invariants of varieties}

Bondal and Orlov started considering $D^b(X)$ as an \emph{invariant} of $X$ with the
following highly influential result:

\begin{thm-non}[Bondal, Orlov, 1997 \cite{bondal-orlov-reconstruction}]
Let $X$ and $Y$ be two smooth, projective varieties with $D^b(X)\cong D^b(Y)$. If $X$
has ample canonical or anti-canonical bundle, then $X\cong Y$.
\end{thm-non}

In other words, at the extreme ends of the curvature spectrum, the derived category
determines the variety. Note the contrast with Mukai's result, which provides examples
of non-isomorphic, derived equivalent varieties with zero curvature (trivial canonical
bundle). This begs the natural question: which (types of) varieties can possibly be
derived equivalent? The philosophy hinted at by the theorems of Mukai, Bondal and Orlov
is not misleading.

\begin{prop-non}
\ Let $X$ and $Y$ be two smooth, projective, derived equivalent varieties.
Then the following hold true:
\begin{enumerate*}
\item[1.\ ] $X$ and $Y$ have the same dimension.
\item[2.\ ] The singular cohomology groups $H^*(X,\IQ)$ and $H^*(Y,\IQ)$ are isomorphic as
      ungraded vector spaces; the same is true for Hoch\-schild cohomology.
\item[3.\ ] If the canonical bundle of $X$ has finite order, then so does the canonical bundle
      of $Y$ and the orders coincide; in particular, if one canonical bundle is trivial,
      then so is the other.
\item[4.\ ] If the canonical (or anti-canonical) bundle of $X$ is ample (or nef), the same is true for $Y$.
\end{enumerate*}
\end{prop-non}

The proposition is the result of the work of many people, see \cite[\S4--6]{huybrechts}.
Stating it here is ahistorical because some of the statements rely on the notion of
Fourier-Mukai transform which we turn to in the next section. It should be said
that our historical sketch is very much incomplete: developments like spaces of stability
conditions \cite{bridgeland-stability} or singularity categories \cite{buchweitz,orlov-sing}
are important but will not play a role here.

\subsection*{Fourier-Mukai partners}

Functors $D^b(X)\to D^b(Y)$ defined by a `kernel', i.e.\ a sheaf (or more
generally, an object) on the product $X\times Y$ were taken up again in the
study of moduli spaces: if a moduli space $M$ of sheaves of a certain
type on $X$ happens to possess a (quasi)universal family $\ke$, then this
family gives rise to a functor $\Coh(M)\to\Coh(X)$, mapping the skyscraper
sheaf in a point $[E]\in M$ to the sheaf it represents, namely $E$.
It was soon realised that its derived functor is a better object of study.
Sometimes, for example, it can be used to show birationality of moduli spaces.

In the following definition, we denote the canonical projections of the product
$X\times Y$ to its factors by $p_X$ and $p_Y$ respectively.

\begin{definition-non}
\ Let $X$ and $Y$ be two smooth, projective varieties and let $K\in D^b(X\times Y)$. The
\emph{Fourier-Mukai functor} with \emph{kernel} $K$ is the composition
\[ \xymatrix@C=2.5em{
    \FM_K\colon D^b(X) \ar[r]^-{p_X^*} & D^b(X\times Y) \ar[r]^-{K\stackrel{\LL}{\otimes}}
                                      & D^b(X\times Y) \ar[r]^-{\RR p_{Y*}} & D^b(Y) } \]
of pullback, derived tensor product with $K$ and derived push-forward.
If $\FM_K$ is an equivalence, then it is called a \emph{Fourier-Mukai transform}.

$X$ and $Y$ are said to be \emph{Fourier-Mukai partners} if a Fou\-rier-\-Mu\-kai
transform exists between their derived categories. The set of all Fourier-Mukai partners
of $X$ up to isomorphisms is denoted by $\FM(X)$.
\end{definition-non}

\begin{remarks-non} \ This important notion warrants a number of comments.

1. Fourier-Mukai functors should be viewed as classical correspondences, i.e.\
maps between cohomology or Chow groups on the level of derived categories. In
particular, many formal properties of correspondences as in \cite[\S14]{fulton} carry
over verbatim: the composition of Fourier-Mukai functors is again such, with the
natural `convoluted' kernel; the (structure sheaf of the) diagonal gives the
identity etc. In fact, a Fourier-Mukai transform induces correspondences on the
Chow and cohomological levels, using the Chern character of the kernel.

2. Neither notation nor terminology is uniform. Some sources mean `Fourier-Mukai
transform' to be an equivalence whose kernel is a sheaf, for example. Notationally,
often used is $\Phi^{X\to Y}_K$ which is inspired by Mukai's original article
\cite{mukai}. This notation, however, has the drawback of being lengthy without
giving additional information in the important case $X=Y$.
\end{remarks-non}

Fourier-Mukai transforms play a very important and prominent role in the theory due to the
following basic and deep result:

\begin{thm-non}[Orlov, 1996 \cite{orlov-k3}]
Given an equivalence $\Phi\colon D^b(X) \isom D^b(Y)$ (as $\IC$-linear, triangulated categories)
for two smooth, projective varieties $X$ and $Y$, then there exists an object $K\in D^b(X\times Y)$
with a functor isomorphism $\Phi\cong\FM_K$. The kernel $K$ is unique up to isomorphism.
\end{thm-non}

By this result, the notions `derived equivalent' and `Fourier-Mukai partners' are
synonymous.

The situation is very simple in dimension 1: two smooth, projective curves are derived
equivalent if and only if they are isomorphic. The situation is a lot more interesting
in dimension 2: apart from the abelian surfaces already covered by Mukai's result,
$\Kthree$ and certain elliptic surfaces can have non-isomorphic FM partners. For $\Kthree$ surfaces,
the statement is as follows (see Section~\ref{sec:K3} for details):

\begin{thm-non}[Orlov, 1996 \cite{orlov-k3}]
For two projective $\Kthree$ surfaces $X$ and $Y$, the following conditions are equivalent:
\begin{enumerate*}
\item[1.\ ] $X$ and $Y$ are derived equivalent.
\item[2.\ ] The transcendental lattices $T_X$ and $T_Y$ are Hodge-isometric.
\item[3.\ ] There exist an ample divisor $H$ on $X$, integers $r\in\IN$, $s\in\IZ$ and a
      class $c\in H^2(X,\IZ)$ such that the moduli space of $H$-semistable sheaves
      on $X$ of rank $r$, first Chern class $c$ and second Chern class $s$ is nonempty,
      fine and isomorphic to $Y$.
\end{enumerate*}
\end{thm-non}

In general, it is a conjecture that the number of FM partners is always finite. For
surfaces, this has been proven by Bridgeland and Maciocia \cite{bridgeland-maciocia}.
The above theorem implies finiteness for abelian varieties, using that an abelian
variety has only a finite number of abelian subvarieties up to isogeny
\cite{golyshev-lunts-orlov}.

\begin{thm-non}[Orlov, Polishchuk 1996, \cite{orlov-abel}, \cite{polishchuk}]
Two abelian varieties $A$ and $B$ are derived equivalent if and only if
$A\times\hat A$ and $B\times\hat B$ are symplectically isomorphic, i.e.\
there is an isomorphism
 $f=\sqmat{\alpha}{\beta}{\gamma}{\delta}\colon A\times\hat A \isom B\times\hat B$
such that $f^{-1}=\sqmat{\hat\delta}{-\hat\beta}{-\hat\gamma}{\hat\alpha}$.
\end{thm-non}

The natural question about the number of FM partners has been studied in greatest
depth for $\Kthree$ surfaces. The first result was shown by Oguiso \cite{oguiso-primes}:
a $\Kthree$ surface with a single primitive ample divisor of degree $d$ has exactly $2^{p(d)-1}$
such partners, where $p(d)$ is the number of prime divisors of $d$. In \cite{HLOY2002},
a formula using lattice counting for general projective $\Kthree$ surfaces was given. In
Section~\ref{sec:K3}, we will reprove this result and give a formula for polarised $\Kthree$
surfaces. We want to mention that FM partners of K3 surfaces have been linked to the
so-called K\"ahler moduli space, see Ma \cite{ma} and Hartmann \cite{hartmann}.

\subsection*{Derived and birational equivalence}

We started this review by motivating derived equivalence as a weakening of
isomorphism, like birationality is. This naturally leads to the question
whether there is an actual relationship between the two notions. At first
glance, this is not the case: since birational abelian varieties are already
isomorphic, Mukai's result provides examples of derived equivalent but not
birationally equivalent varieties. And in the other direction, let $Y$ be the
blowing up of a smooth projective variety $X$ of dimension at least two in a
point. Then $X$ and $Y$ are obviously birationally equivalent but never
derived equivalent by a result of Bondal and Orlov \cite{bondal-orlov-semi}.

Nevertheless some relation is expected. More precisely:

\begin{conj-non}[Bondal, Orlov \cite{bondal-orlov-semi}]
If $X$ and $Y$ are smooth, projective, birationally equivalent varieties with trivial
canonical bundles, then $X$ and $Y$ are derived equivalent.
\end{conj-non}

Kawamata suggested a generalisation using the following notion: two smooth,
projective varieties $X$ and $Y$ are called \emph{$K$-equivalent} if there
is a birational correspondence $X \antimaps{p} Z \maps{q} Y$
with $p^*\omega_X \cong q^*\omega_Y$. He conjectures that K-equivalent varieties
are D-equivalent.

The conjecture is known in some cases, for example
the standard flop (Bondal, Orlov \cite{bondal-orlov-semi}),
the Mukai flop (Kawamata \cite{kawamata-dk}, Namikawa \cite{namikawa}),
Calabi-Yau threefolds (Bridgeland \cite{bridgeland-flops}) and
Hilbert schemes of $\Kthree$ surfaces (Ploog \cite{ploog}).


\section{Lattices} \label{sec:lattices}

\noindent
Since the theory of $\Kthree$ surfaces is intricately linked to lattices, we provide
a review of the lattice theory as needed in this note. By a lattice we always mean a
free abelian group $L$ of finite rank equipped with a non-degenerate symmetric bilinear
pairing $\pp{L}\colon L\times L\to\IZ$. The lattice $L$ is called \emph{even} if
$(v,v)\in 2\IZ$ for all $v\in L$. We shall assume all our lattices to be even.

Sometimes, we denote by $L_K$ the $K$-vector space $L\otimes K$, where $K$ is a field
among $\IQ, \IR, \IC$. The pairing extends to a symmetric bilinear form on $L_K$.
The \emph{signature} of $L$ is defined to be that of $L_\IR$.

The lattice $L$ is called \emph{unimodular} if the
canonical homomorphism $d_L\colon L\to L^{\vee}=\Hom(L,\IZ)$ with
$d_L(v)=(v,\cdot)$ is an isomorphism. Note that $d_L$ is always injective,
as we have assumed $\pp{L}$ to be non-degenerate. This implies that for every element
$f\in L^\vee$ there is a natural number $a\in\IN$ such that $af$ is in the image of
$d_L$.
Thus $L^\vee$ can be identified with the subset
 $\{w\in L\otimes\IQ \mid (v,w)\in\IZ ~ \forall v\in L\}$
of $L\otimes\IQ$ with its natural $\IQ$-valued pairing.
%

We shall denote the \emph{hyperbolic plane} by $U$. A \emph{standard basis} of $U$ is
a basis $e,f$ with $e^2=f^2=0$ and $(e,f)=1$. The lattice $E_8$ is the unique
positive definite even unimodular lattice of rank $8$, and we denote by $E_8(-1)$ its
negative definite opposite.
For an integer $n\neq0$ we denote by $\langle n\rangle$ the rank one lattice where both
generators square to $n$. Finally, given a lattice $L$, then $aL$
denotes a direct sum of $a$ copies of the lattice $L$.

Given any subset $S\subseteq L$, the \emph{orthogonal complement} is
$S\orth:=\{v\in L\mid(v,S)=0\}$. A sublattice $S\subseteq L$ is called
\emph{primitive} if the quotient group $L/S$ is torsion free. Note the
following obvious facts: $S\orth\subseteq L$ is always a primitive sublattice;
we have $S\subseteq S^{\perp\perp}$; and $S$ is primitive if and only if
$S=S^{\perp\perp}$. In particular, $S^{\perp\perp}$ is the \emph{primitive hull}
of $S$.

A vector $v\in L$ is called \emph{primitive} if the lattice $\IZ v$ generated by
it is primitive.

The \emph{discriminant group} of a lattice $L$ is the finite abelian group
$D_L=L^{\vee}/L$. Since we have assumed $L$ to be even it carries a natural
quadratic form $q_L$ with values in $\IQ / 2\IZ$. By customary abuse of
notation we will often speak of a quadratic form $q$ (or $q_L$), suppressing
the finite abelian group it lives on.
Finally, for any lattice $L$, we denote by $l(L)$ the minimal number of generators
of $D_L$.

\subsection*{Gram matrices}
We make the above definitions more explicit using the matrix description.
After choosing a basis, a lattice on $\IZ^r$ is given by a symmetric $r\times r$
matrix $G$ (often called Gram matrix), the pairing being $(v,w)=v^tGw$ for
$v,w\in\IZ^r$. To be precise, the $(i,j)$-entry of $G$ is $(e_i,e_j)\in\IZ$
where $(e_1,\dots,e_r)$ is the chosen basis.

Changing the matrix by symmetric column-and-row operations gives an isomorphic
lattice; this corresponds to $G\mapsto SGS^t$ for some $S\in\text{SL}(r,\IZ)$.
Since our pairings are non-degenerate, $G$ has full rank. The lattice is
unimodular if the Gram matrix has determinant $\pm1$. It is even if and only if
the diagonal entries of $G$ are even.

The inclusion of the lattice into its dual is the map $G\colon\IZ^r\embed\IZ^r$,
$v\mapsto v^tG$. Considering a vector $\varphi\in\IZ^r$ as an element of the dual
lattice, there is a natural number $a$ such that $a\varphi$ is in the image of $G$,
i.e.\ $v^tG=a\varphi$ for some integral vector $v$. 
Then $(\varphi,\varphi)=(v,v)/a^2\in\IQ$.

The discriminant group is the finite abelian group with presentation matrix $G$,
i.e.\ $D\cong\IZ^r/\image(G)$. Elementary operations can be used to diagonalise it.
The quadratic form on the discriminant group is computed as above, only now taking
values in $\IQ/2\IZ$.

The \emph{discriminant} of $L$ is defined as the order of the discriminant group.
It is the absolute value of the determinant of the Gram matrix:
 $\disc(L):=\#D_L=|\det(G_L)|$. Classically, discriminants (of quadratic forms)
are defined with a factor of $\pm1$ or $\pm1/4$; see Example~\ref{ex:binaryforms}.

\subsection*{Genera}
Two lattices $L$ and $L'$ of rank $r$ are said to be \emph{in the same genus}
if they fulfill one of the following equivalent conditions:
\begin{enumerate}
\item[(1)\ ] The localisations $L_p$ and $L'_p$ are isomorphic for all primes $p$
      (including $\IR$).
\item[(2)\ ] The signatures of $L$ and $L'$ coincide and the discriminant forms
      are isomorphic: $q_L \cong q_{L'}$.
\item[(3)\ ] The matrices representing $L$ and $L'$ are \emph{rationally equivalent without
      essential denominators}, i.e.\ there is a base change in $\text{GL}(r,\IQ)$
      of determinant $\pm1$, transforming $L$ into $L'$ and whose denominators are
      prime to $2\cdot\disc(L)$.
\end{enumerate}
For details on localisations, see \cite{nikulin}. The equivalence of (1) and
(2) is a deep result of Nikulin (\cite[1.9.4]{nikulin}).
We elaborate on (2): a map $q\colon A\to\IQ/2\IZ$ is called a quadratic form on
the finite abelian group $A$ if $q(na)=n^2q(a)$ for all $n\in\IZ$, $a\in A$ and
if there is a symmetric bilinear form $b\colon A\times A\to\IQ/\IZ$ such that
$q(a_1+a_2)=q(a_1)+q(a_2)+2b(a_1,a_2)$ for all $a_1,a_2\in A$. It is clear that
discriminant forms of even lattices satisfy this definition.
Two pairs $(A,q)$ and $(A',q')$ are defined to be isomorphic if there is a group
isomorphism $\varphi\colon A\isom A'$ with $q(a)=q'(\varphi(a))$ for all $a\in A$.

The history of the equivalence between (1) and (3) is complicated: Using
analytical methods, Siegel \cite{siegel} proved that $L$ and $L'$ are in the
same genus if and only if for every positive integer $d$ there exists a rational
base change $S_d\in\text{GL}(r,\IQ)$ carrying $L$ into $L'$ and such that the
denominators of $S_d$ are prime to $d$ (and he called this property rational
equivalence without denominators).
There are algebraic proofs of that statement, e.g.\ \cite[Theorem~40]{jones} or
\cite[Theorem~50]{watson}. These references also contain (3) above, i.e.\ the
existence of a single $S\in\text{GL}(r,\IQ)$ whose denominators are prime to
$2\cdot\disc(L)$.

For binary forms, all of this is closely related to classical number theory.
In particular, the genus can then also be treated using the ideal class group of
quadratic number fields. See \cite{cox} or \cite{zagier} for this. Furthermore,
there is a strengthening of (3) peculiar to \emph{field discriminants} (see
\cite[\S3.B]{cox}):
\begin{enumerate*}
\item[(4)\ ] Let $L=\sqmat{2a}{b}{b}{2c}$ and $L'=\sqmat{2a'}{b'}{b'}{2c'}$ be two
           binary even, indefinite lattices with $\gcd(2a,b,c)=\gcd(2a',b',c')=1$
           and of the same discriminant $b^2-4ac=:D$ such that either
           $D\equiv 1 \mod 4$, $D$ squarefree, or $D=4k$, $k\not\equiv 1 \mod 4$,
           $k$ squarefree. Then $L$ and $L'$ are in the same genus if and
           only if they are rationally equivalent, i.e.\ there is a base change
           $S\in\text{GL}(2,\IQ)$ taking $L$ to $L'$.
\end{enumerate*}

The genus of $L$ is denoted by $\kg(L)$ and it is a basic, but non-trivial fact that $\kg(L)$
is a finite set.
Given natural numbers $t_+$ and $t_-$ and a finite abelian group $D_q$ together with a
quadratic form $q\colon D_q\to\IQ/2\IZ$ , we also denote by $\kg(t_+,t_-,q)$ the set of
even lattices of this genus.

\begin{example} \label{ex:binaryforms}
We consider binary forms, that is lattices of rank 2. Clearly, a symmetric bilinear
form with Gram matrix $\sqmat{a}{b}{b}{c}$ is even if and only if both diagonal terms
are even.

Note that many classical sources use quadratic forms instead of lattices. We explain
the link for binary forms $f(x,y)=ax^2+bxy+cy^2$ (where $a,b,c\in\IZ$). The associated
bilinear form has Gram matrix $G=\frac{1}{2}\sqmat{2a}{b}{b}{2c}$ --- in particular, it
need not be integral. An example is $f(x,y)=xy$. In fact, the bilinear form, i.e.\ $G$,
is integral if and only if $b$ is even (incidentally, Gau\ss\ always made that assumption).
Note that the quadratic form $2xy$ corresponds to our hyperbolic plane $\sqmat{0}{1}{1}{0}$.
The discriminant of $f$ is classically defined to be $D:=b^2-4ac$ which differs from our
definition (i.e.\ $|\det(G)|=\#D$) by a factor of $\pm4$.

We proceed to give specific examples of lattices as Gram matrices.
Both $A=\sqmat{2}{4}{4}{0}$ and $B=\sqmat{0}{4}{4}{0}$ are indefinite, i.e.\ of
signature $(1,1)$, and have discriminant 16, but the discriminant groups are not
isomorphic: $D_A=\IZ/2\IZ\times\IZ/8\IZ$ and $D_B=\IZ/4\IZ\times\IZ/4\IZ$. Thus $A$
and $B$ are not in the same genus.

Another illuminating example is given by the forms $A$ and $C=\sqmat{-2}{4}{4}{0}$.
We first notice that these forms are not isomorphic: the form $A$ represents $2$, but $C$ does not,
as can be seen by looking at the possible remainders of $-2x^2+8xy$ modulo $8$.
The two forms have the same signature and discriminant groups, but the discriminant forms are different.
To see this we note that $D_A$ is generated by the residue classes of $t_1=e_1/2$ and
$t_2=(2e_1 + e_2)/8$, whereas $D_C$ is generated by the residue classes of $s_1=e_1/2$ and
$s_2=(-2e_1 + e_2)/8$. The quadratic forms $q_A$ and $q_C$ are determined by
$q_A(\overline{t_1})=1/2$, $q_A(\overline{t_2})= 3/8$ and
$q_C(\overline{s_1})=-1/2$, $q_C(\overline{s_2})= -3/8$. The forms cannot be isomorphic, for the
subgroup of $D_A$ of elements of order 2 consists of $\{0,t_1,4t_2,t_1+4t_2\}$ (this is the Klein
four group) and the values of $q_A$ on these elements in $\IQ/2\IZ$ are $0, 1/2, 4^2\cdot 3/8=0, 42/4=1/2$.
Likewise, the values of $q_C$ on the elements of order 2 in $D_C$ are $0$ and $-1/2$.
Hence $(D_A,q_A)$ and $(D_C,q_C)$ cannot be isomorphic.

Zagier's book also contains the connection of genera to number theory and their classification using
ideal class groups \cite[\S8]{zagier}.
An example from this book \cite[\S12]{zagier} gives an instance of lattices
in the same genus which are not isomorphic: the forms $D=\sqmat{2}{1}{1}{12}$ and
$E=\sqmat{4}{1}{1}{6}$ are positive definite of field discriminant $-23$. They are in the
same genus (one is sent to the other by the fractional basechange
$-\frac{1}{2}\sqmat{1}{1}{-3}{1}$) but not equivalent: $D$ represents $2$ as
the square of $(1,0)$ whereas $E$ does not represent 2 as $4x^2+2xy+6y^2=3x^2+(x+y)^2+5y^2\geq4$
if $x\neq0$ or $y\neq0$.
\end{example}

Unimodular, indefinite lattices are unique in their genus, as follows from their well known
classification. A generalisation is given by \cite[Cor.\ 1.13.3]{nikulin}:
\begin{lemma} \textbf{\emph{(Nikulin's criterion)}} An indefinite
lattice $L$ with $\rk(L)\geq2+l(L)$ is unique within its genus.
\end{lemma}
Recall that $l(L)$ denotes the minimal number of generators of the finite group $D_L$.
Since always $\rk(L)\geq l(L)$, Nikulin's criterion only fails to apply in two cases, namely if
$l(L)=\rk(L)$ or $l(L)=\rk(L)-1$.


\bigskip

For a lattice $L$, we denote its group of isometries by $\OO(L)$. An isometry of
lattices, $f\colon L\isom L'$ gives rise to $f_\IQ\colon L_\IQ\isom L'_\IQ$ and
hence to $D_f\colon D_L\isom D_{L'}$. In particular, there is a natural homomorphism
$\OO(L)\to\OO(D_L)$ which is used to define the \emph{stable isometry group} as
\[ \tilde\OO(L) := \ker(\OO(L)\to\OO(D_L)) .\]

Finally, we state a well known result of Eichler, \cite[\S10]{eichler}:
\begin{lemma} \textbf{\emph{(Eichler's criterion)}}
Suppose that $L$ contains $U\oplus U$ as a direct summand. The
$\OO(L)$-orbit of a primitive vector $v\in L$ is determined by the length
$v^2$ and the element $v/\div(v)\in D(L)$ of the discriminant group.
\end{lemma}


\section{Overlattices} \label{sec:overlattices}

\noindent
In this section, we elaborate on Nikulin's theory of overlattices and primitive embeddings \cite{nikulin}; we also give some examples. Eventually, we generalise slightly to cover a setting needed for the Fourier-Mukai partner counting in the polarised case.

We fix a lattice $M$ with discriminant form $q_M\colon D_M\to\IQ/2\IZ$.

By an \emph{overlattice} of $M$ we mean a lattice embedding $i\colon M\embed L$
 with $M$ and $L$ of the same rank. Note that we have inclusions
\[ \xymatrix@1{
   M \ar[r]^i \ar@/_1pc/[rrr]_{d_M} & L \ar[r]^-{d_L} & L^\vee \ar[r]^-{i^\vee} & M^\vee
} \]
with $d_L\colon L\embed L^\vee$ and $d_M\colon M\embed M^\vee$ the canonical maps. (For now,
we will denote these canonical embeddings just by $d$, and later not denote them at all.)
From this we get a chain of quotients
\[ \xymatrix@1{
   L/iM \ar[r]^-{d} & L^\vee/diM \ar[r]^-{i^\vee} & M^\vee/i^\vee diM = D_M .
} \]
We call the image $H_i\subset D_M$ of $L/iM$ the \emph{classifying subgroup} of the
overlattice. Note that $D_M$ is equipped with a quadratic form, so we can also speak
of the orthogonal complement $H_i^\perp$.
We will consider $L^\vee/diM$ as a subgroup of $D_M$ in the same way via $i^\vee$.

We say that two ebeddings $i\colon M\embed L$ and $i'\colon M\embed L'$ {\em define the same overlattice} 
if there is an isometry $f\colon L\isom L'$ such that $fi=i'$:

\[ \xymatrix{
 M \ar[r]^i \ar@{=}[d] & L \ar[d]^f \\
 M \ar[r]^{i'}         & L'
} \]

This means in particular that within each isomorphism class, we can restrict to looking at
embeddings $i\colon M\embed L$ into a \emph{fixed} lattice $L$.

\begin{lemma}\emph{\cite[Proposition 1.4.1]{nikulin}} \label{lem:nikulinI}
Let $i\colon M\embed L$ be an overlattice. Then the subgroup $H_i$ is isotropic
in $D_M$, i.e.\ $q_M|_{H_i}=0$. Furthermore, $H_i^\perp=L^\vee/diM$ and there is a
natural identification $H_i^\perp/H_i\cong D_L$ with $q_M|_{H_i^\perp/H_i}=q_L$.
\end{lemma}

We introduce the following sets of overlattices $L$ of $M$ and quotients $L/M$ respectively,
where we consider $L/M$ as an isotropic subgroup of the discriminant group $D_M$:
\begin{align*}
\ko(M) &:= \{ (L,i) \mid L \text{ lattice}, i\colon M\embed L \text{ overlattice} \}
          &&
\\
\kq(M) &:= \{ H\subset D_M \text{ isotropic} \}.
           &&
\end{align*}
We also use the notation $\ko(M,L)$ to specify that the bigger lattice is isomorphic to $L$. With this notation
we can write $\ko(M)$ as a disjoint union
$$
\ko(M) = \coprod_L \ko(M,L)
$$
where $L$ runs through all isomorphism classes of possible overlattices of $M$.

\begin{remark} \label{rem:isometry-extension}
The set $\kq(M)$ is obviously finite. On the other hand, an overlattice $i\colon M\embed L$
can always be modified by an isometry $f\in\OO(M)$ to yield an overlattice $if\colon M\embed L$.
However, if $f\in\tilde\OO(M)$ is a stable isometry, then it can be extended to an isometry
of $L$ and hence $i$ and $if$ define the same overlattice. This shows that $\ko(M)$ is also finite.
\end{remark}

The following lemma is well known and implicit in \cite{nikulin}.

\begin{lemma}
There is a bijection between $\ko(M)$ and $\kq(M)$.
\end{lemma}

\begin{proof} \
We use the maps
\begin{align*}
\mapH &\colon \ko(M) \to \kq(M), \quad (L,i) \mapsto H_i ,    \\
\mapL &\colon \kq(M) \to \ko(M), \quad H     \mapsto (L_H,i_H)
\end{align*}
where, for $H\in\kq(M)$, we define
 $ L_H := \{ \varphi \in M^\vee\mid [\varphi] \in H\} = \pi\inv(H)$
where $\pi\colon M^\vee\to D_M$ is the canonical projection.
The canonical embedding $d\colon M\embed M^\vee$ factors through $L_H$,
giving an injective map $i_H\colon M\to L_H$.
All of this can be summarised in a commutative diagram of short exact sequences
\[ \xymatrix{
0 \ar[r] & M \ar[r]^{i_H} \ar@{=}[d] & L_H    \ar[r] \ar@{^(->}@<-.7ex>[d] & H   \ar[r] \ar@{^(->}[d] & 0 \\
0 \ar[r] & M \ar[r]                 & M^\vee \ar[r]^\pi            & D_M \ar[r] & 0 .
} \]

The abelian group $L_H$ inherits a $\IQ$-valued form from $M^\vee$.
This form is actually $\IZ$-valued because of $q_M|_H=0$. Furthermore,
the bilinear form on $L_H$ is even since the quadratic form on $D_M$ is
$\IQ/2\IZ$-valued. Hence, $L_H$ is a lattice and $i_H$ is obviously a
lattice embedding.

It is immediate that $\mapH\mapL=\id_{\kq(M)}$. On the other hand, the overlattices
$\mapL\mapH(L,i)$ and $(L,i)$ are identified by the embedding $L\to M^\vee$,
$v\mapsto \langle v,i(\cdot)\rangle_M$ which has precisely $\mapL\mapH(L,i)$
as image.
\qed
\end{proof}

\noindent
We want to refine this correspondence slightly. For this we fix a quadratic form $(D,q)$ which occurs as the discriminant
of some lattice (and forget $L$) and set
\begin{align*}
\ko(M,q) &:= \{ (L,i) \in \ko(M) \mid [L]\in\kg(\sign(M),q) \}, \\
\kq(M,q) &:= \{ H     \in \kq(M) \mid q_M|_{H^\perp/H}\cong q \}.
\end{align*}
The condition $q_{M}|_{H^\perp/H}\cong q$ here includes $H^\perp/H\cong D$.

\begin{lemma}
There is a bijection between $\ko(M,q)$ and $\kq(M,q)$.
\end{lemma}

\begin{proof} \ 
We only have to check that the maps $\mapH\colon\ko(M,q)\to\kq(M)$ and
$\mapL\colon\kq(M,q)\to\ko(M)$ have image in $\kq(M,q)$ and $\ko(M,q)$,
respectively. For $\mapH$, this is part of Lemma~\ref{lem:nikulinI}.
For $\mapL$, we have $\sign(L_H)=\sign(M)$ and the discriminant form of $L_H$ is
$D_M|_{H\orth/H}\cong q$, by assumption on $H$.
\qed \end{proof}

In the course of our discussions we have to distinguish carefully between 
different notions equivalence of lattice embeddings.
The following notion is due to Nikulin (\cite[Proposition 1.4.2]{nikulin}):

\begin{definition} Two embeddings $i,i'\colon M\embed L$ define \emph{isomorphic overlattices},
denoted $i\simeq i'$, if there exists an isometry $f\in\OO(L)$ with $fi(M)=i'(M)$ --- inducing 
an isometry $f|_M\in\OO(M)$ --- or, equivalently if there is a commutative diagram:
\[ \xymatrix{
 M \ar[r]^i \ar[d]^{f|_M} & L \ar[d]^f \\
 M \ar[r]^{i'}           & L
} \]
\end{definition}
Note that this definition also makes sense if $M$ and $L$ do not
necessarily have the same rank. Two embeddings of lattices $i,i'\colon M\embed L$ of the same rank defining 
the \emph{same} overlattice are in particular isomorphic.

\begin{definition} 
Two embeddings $i,i'\colon M\embed L$ are \emph{stably isomorphic}, denoted $i\approx i'$,
if there exists a stable isometry $f\in\tilde\OO(L)$ with $fi(M)=i'(M)$, i.e.\ there is a commutative diagram
\[ \xymatrix{
 M \ar[r]^i \ar[d]^{f|_M} & L \ar[d]^{f \text{ stable}} \\
 M \ar[r]^{i'}         & L
} \]
\end{definition}
We note that embeddings of lattices of the same rank defining the same overlattice are not 
necessarily stably isomorphic. 

We can put this into a broader context. For this we consider the set 
\[
\ke(M,L):= \{ i\colon M \embed L \} 
\]
of embeddings of $M$ into $L$ where, for the time being, we do  not assume $M$ and $L$ to have the same rank.
The group $\OO(M) \times \OO(L)$ acts on this set by $(g,\tilde g)\colon i \mapsto \tilde{g} i g^{-1}$.
Instead of the action of $\OO(M) \times \OO(L)$ on $\ke(M,L)$ one can also consider the action of any subgroup
and we shall see specific examples later when we discuss Fourier-Mukai partners of $\Kthree$ surfaces.
If $M$ and $L$ have the same rank, then the connection with our previously considered equivalence
relations is the following:
\[
\ko(M,L)= (\{\id_M\} \times \OO(L)) \backslash  \ke(M,L).
\]
The set of all isomorphic overlattices of $M$ isomorphic to $L$ is given by 
$(\OO(M) \times \OO(L)) \backslash  \ke(M,L)$ whereas stably isomorphic embeddings are given by
$(\OO(M) \times \tilde\OO(L)) \backslash  \ke(M,L)$.




We now return to our previous discussion of the connection between overlattices and isotropic subgroups.

\begin{lemma} \label{lem:overlattices}
Let $i,i' \colon M \embed L$ be embeddings of lattices of the same rank. 
Then $i\simeq i'$ if and only if there exists an isometry $g\in\OO(M)$ such that $D_g(H_i)=H_{i'}$.
\end{lemma}

\begin{proof} \ 
Given $f\in\OO(L)$ with $fi(M)=i'(M)$, then $g:=f|_M$ will have the correct property.

Given $g$, recall that the lattices are obtained from their classifying subgroups as
$\pi^{-1}(H_i)$ and $\pi^{-1}(H_{i'})$. Then, $D_g(H_i)=H_{i'}$ implies that the map
$g^\vee\colon M^\vee\isom M^\vee$ induced from $g$ sends $L$ to itself, and $f=g^\vee|_L$ gives the
desired isomorphism.
\qed \end{proof}

Note that an isometry $g\in\OO(M)$ with $D_g(H_i)=H_{i'}$ induces an isomorphism
$H_i^\perp\isom H_{i'}^\perp$ and hence an isomorphism of the quotients. Recall that
there is a natural identification $H_i^\perp/H_i=D_L$.

\begin{lemma}
$i\approx i'$ if and only if there exists an isometry $g\in\OO(M)$ such that $D_g(H_i)=H_{i'}$
and the induced map $D_L=H_i^\perp/H_i \to H_{i'}^\perp/H_{i'}=D_L$ is the identity.
\end{lemma}

\begin{proof} \ 
Just assuming $D_g(H_i)=H_{i'}$, we get a commutative diagram
\[ \xymatrix@C=3em{
   H_i    \ar@{^{(}->}[r] \ar[d]_{D_g} &
   H_i^\perp \ar@{->>}[r]^{(i^\vee)^{-1}} \ar[d]_{D_g} &
   D_L \ar[d]^{D_f}
\\
   H_{i'} \ar@{^{(}->}[r] &
   H_{i'}^\perp \ar@{->>}[r]^{(i'^{\vee})^{-1}} &
   D_L
} \]
which, together with the proof of Lemma~\ref{lem:overlattices} shows the claim.
\qed \end{proof}


\subsection*{Overlattices from primitive embeddings}

A natural source of overlattices is $M:=T\oplus T\orth \subset L$ for any
sublattice $T\subset L$. If $T$ is moreover a primitive sublattice of $L$, then
the theory sketched above can be refined, as we explain next. We start with an elementary
lemma.

\begin{lemma} \label{lem:primitive}
Let $A,B\subset L$ be two sublattices such that $i\colon A\oplus B \embed L$ is an
overlattice, i.e.\ $A$ and $B$ are mutually orthogonal and $\rk(A\oplus B)=\rk(L)$.
Then $p_A \colon H_i\embed D_{A\oplus B}\surject D_A$ is injective if
and only if $B$ is primitive in $L$.
\end{lemma}

\begin{proof} \ 
The commutative diagram with exact rows
\[ \xymatrix{
0 \ar[r] & A\oplus B \ar[r] \ar[d] & L \ar[r] \ar[d] & H_i \ar[r] \ar[d] & 0 \\
0 \ar[r] & A         \ar[r]        & A^\vee \ar[r]    & D_A \ar[r] & 0
} \]
leads to the following short exact sequence of the kernels:
\[ 0 \to B \to B^{\perp\perp} \to \ker(p_A) \to 0 \]
(note that the kernel of the map $L\to A^\vee, v\mapsto \langle v,\cdot\rangle|_A$
is the primitive hull of $B$). Hence $p_A$ is injective if and only if $B=B^{\perp\perp}$,
i.e.\ $B$ is a primitive sublattice.
\qed \end{proof}

\begin{example}
We consider the rank $2$ lattice $L$ with Gram matrix $\sqmat{2}{0}{0}{2}$; let $e_1, e_2$ be an orthogonal basis, so that $e_1^2=e_2^2=2$. With $T=\langle 8\rangle$ having basis $2e_1$
and $K:=T\orth$, we get that $H_i\to D_T$ is injective whereas $H_i\to D_K$ is not.
\end{example}

Let $j_T\colon T\embed L$ be a sublattice and $K:=T\orth$ its orthogonal complement
with embedding $j_K\colon K\embed L$. By Lemma~\ref{lem:nikulinI}, the overlattice
 $i:=j_T\oplus j_K \colon T\oplus K \embed L$
corresponds to the isotropic subgroup
 $H_i\subset D_{T\oplus K}$.

By Lemma~\ref{lem:primitive}, the map $p_T \colon H_i \embed D_{T\oplus K}\surject D_T$
is always injective, since $K\subset L$ is an orthogonal complement, hence primitive.
The map $p_K \colon H_i\embed D_{T\oplus K}\surject D_K$ is injective if and only if
$T\subset L$ is a primitive sublattice.

If $j_T\colon T\embed L$ is primitive, then $\Gamma_i:=p_T(H_i)\subseteq D_T$ is a
subgroup such that there is a unique, injective homomorphism
$\gamma_i\colon \Gamma_i\to D_K$. The image of $\gamma_i$ is $p_K(H_i)$ and its graph
is $H_i$.

For fixed lattices $L$, $K$, $T$ 
we introduce the follwoing sets
\begin{align*}
\kp(T,L)   &:= \{ j_T\colon T\embed L \text{ primitive} \}, \\
\kp(T,K,L) &:= \Big\{ (j_T,j_K) \:| \begin{array}{l} j_T \in \kp(T,L), j_K \in \kp(K,L) \\
                                                     j_T\oplus j_K\in\ke(T\oplus K,L)
                                    \end{array} \Big\}.
\end{align*}
As in the previous section we can consider various notions of equivalence on the set 
$\kp(T,K,L)$ by considering the action of suitable subgroups of $\OO(T) \times \OO(K) \times \OO(L)$.
Since we are only interested in overlattices in this section we shall assume for the rest of
this section that
\begin{assumption}
$\qquad\qquad \rk (T) + \rk(K) = \rk(L)$.
\end{assumption}
In the previous section we said that two embeddings define the same overlattice if they differ by the action
of $\{\id_T\} \times \{\id_K\} \times \OO(L)$ and accordingly we set
\[
\ko(T,K,L) = (\{\id_T\} \times \{\id_K\} \times \OO(L)) \backslash \kp(T,K,L) .
\]

We now also consider a quadratic form $(D,q)$ which will play the role of the discriminant of the overlattice. 
Choosing a representative $L$ for each element in $\kg(\sign(T\oplus K),q)$,
we also introduce the equivalents of the sets of the previous section:
\begin{align*}
\kp(T,K,q) &:= \Big\{ (L,j_T,j_K) \:| \begin{array}{l} [L]\in\kg(\sign(T\oplus K),q), \\
                                                       (j_T,j_K)\in\kp(T,K,L)
                                      \end{array} \Big\}, \\
\kq(T,K,q) &:= \{ H \in \kq(T\oplus K,q) \mid p_T|_H \text{ and } p_K|_H \text{ are injective} \} .
\end{align*}
Dividing out by the action of the overlattice we also consider $\ko(T,K,q)$.
The condition in the definition of $\kq(T,K,q)$ means that $H$ is the graph of an
injective group homomorphism $\gamma\colon\Gamma\embed D_K$ with $\Gamma:=p_T(H)$
and $\image(\gamma)=p_K(H)$. Note that $q_{T\oplus K}|_H=0$ is equivalent to
$q_K\gamma=-q_T|_\Gamma$.

Evidently, $\kp(T,K,q)$, respectively $\ko(T,K,q)$, is the disjoint union of $\kp(T,K,L)$, 
respectively $\ko(T,K,L)$ over representative lattices $L$ of the genus prescribed by 
$\sign(T\oplus K)$ and discriminant form $q$.
The difference between $\kp(T,K,q)$ and $\kp(T,K,L)$ is that the former set does not
specify the overlattice but just its genus and we need $\kp(T,K,q)$ because we are
interested in describing lattices by discriminant forms, but the latter only see the
genus.


\begin{lemma} \label{lem:nikulin151}
For $T$, $K$ and $q$ as above, the sets $\ko(T,K,q)$ and $\kq(T,K,q)$ are in bijection.
\end{lemma}
\begin{proof} \ 
The main idea is that the restrictions of $\mapH$ and $\mapL$ to the newly introduced 
sets factor as follows
\[ \xymatrix{
 \ko(T,K,q) \ar@{^{(}->}[r] \ar@{-->}@<1ex>[d] & \ko(T\oplus K,q) \ar@<1ex>[d]^\mapH &
                             (L,j_T,j_K) \mapsto (L,j_T\oplus j_K) \\
 \kq(T,K,q) \ar@{^{(}->}[r] \ar@{-->}@<1ex>[u] & \kq(T\oplus K,q) \ar@<1ex>[u]^\mapL &
                             H \mapsto H
} \]

Indeed, the map $\mapH|_{\kp(T,K,q)}$ factors via $\kq(T,K,q)$ in view of Lemma~\ref{lem:primitive}.

In order to see that $\mapL|_{\kq(T,K,q)}$ factors over $\ko(T,K,q)$, we take an
isotropic subgroup $H\subset D_{T\oplus K}$. Then we can form the overlattice
$L_H=\pi^{-1}(H)$ of $T\oplus K$. Obviously, this gives embeddings
$j_T\colon T\embed L_H$ and $j_K\colon K\embed L_H$. These are primitive since the
projections $H\to p_T(H)$ and $H\to p_K(H)$ are isomorphisms. Next, the sublattices
are orthogonal to each other:
 $j_T\colon T\to T^\vee\oplus K^\vee, v\mapsto (\langle v,\cdot\rangle,0)$ and
 $j_K\colon K\to T^\vee\oplus K^\vee, w\mapsto (0,\langle w,\cdot\rangle)$.
Finally, they obviously span $L_H$ over $\IQ$.
\qed \end{proof}

Fix a subgroup $G_T\subseteq\OO(T)$. Two pairs $(L,i,j),(L',i',j')\in\kp(T,K,q)$
are called \emph{$G_T$-equivalent} if there is an  isometry
$\varphi \colon L \cong L'$ such that $\varphi(iT)=i'T$ and
 $\varphi_T:=(i')^{-1} \circ \varphi|_{i(T)} \circ i\in G_T$
for the induced isometry of $T$.


\begin{lemma} \emph{\cite[1.15.1]{nikulin}} \label{lem:G-eq1}
Let $H,H'\in\kq(T,K,q)$. Then $\mapL(H)$ and $\mapL(H')$ are $G_T$-equivalent
if and only if there is $\psi\in G_T\times\OO(K)$ such that $D_\psi(H)=H'$.
\end{lemma}

\begin{proof} \ 
First note that the condition $D_\psi(H)=H'$ is equivalent to the one in
\cite{nikulin}: there are $\psi_T\in G_T$ and $\psi_K\in\OO(K)$ such that
$D_{\psi_T}(\Gamma)=\Gamma'$ and $D_{\psi_K}\gamma=\gamma'D_{\psi_T}$ where
$H$ and $H'$ are the graphs of $\gamma\colon \Gamma\to D_K$ and
$\Gamma'\colon H'\to D_K$, respectively.

Suppose that $(L,i,j)$ and $(L',i',j')$ are $G_T$-equivalent. Thus there is an isometry
$\varphi L \colon L'$ with $\varphi(iT)=i'T$. In particular,
$\varphi(i(T)\orth_L)=i'(T)\orth_{L'}$; using the isomorphisms $j$ and $j'$
we get an induced isometry $\varphi_K\in\OO(K)$.
We have established the following commutative diagram with exact rows
\[ \xymatrix{
0 \ar[r] & iT\oplus jK   \ar[r] \ar[d]^{\varphi_T\oplus\varphi_K} & L \ar[r] \ar[d]^\varphi
         & L/(iT\oplus jK) \ar[r] \ar@{..>}[d]^{D_\varphi} & 0 \\
0 \ar[r] & i'T\oplus j'K \ar[r] & L' \ar[r] & L'/(i'T\oplus j'K) \ar[r] & 0
} \]
Put $\psi:=(\varphi_T,\varphi_K)\in G_T\times\OO(K)$.
Using the identification of $L/(iT\oplus jK)$ with $H\subset D_{T\oplus K}$
obtained from $i$ and $j$ (and analogously for $H'$), the isomorphism
$D_\varphi$ on discriminants turns into the isomorphism $D_\psi\colon H\isom H'$.
Note that by construction $\psi^\vee|_L=\varphi$.

Given $\psi\in G_T\times\OO(K)$, consider the induced isomorphism on the dual
$\psi^\vee\colon (T\oplus K)^\vee\isom (T\oplus K)^\vee$.
By the assumption $D_\psi(H)=H'$, this isomorphism restricts to
$\varphi:=\psi_T^\vee\oplus\psi_K^\vee|_{L_H}\colon L_H \isom L_{H'}$.
Finally, under the embeddings $i_H, j_H, i_{H'}, j_{H'}$ the
induced isometries of $\varphi$ combine to $(\varphi_T,\varphi_K)=\psi$.
\qed \end{proof}

\begin{assumption} \label{assumption}
From now on we suppose that the embedding lattice is uniquely determined by the
signature (derived from $T\oplus K$) and the discriminant form $q$. In other words,
we postulate that there be a single lattice $L$ in that genus, i.e.\
$\kp(T,K,L)=\kp(T,K,q_L)$.
\end{assumption}

We say that two primitive embeddings $(i,j),(i',j')\in\kp(T,K,L)$ are \emph{$G_L$-equivalent} if
there is an isometry $f\in G_L$ such that $f(iT\oplus jK)=i'T\oplus j'K$.

Combining the two isometry subgroups $G_T\subset\OO(T)$ and $G_L\subset\OO(L)$, we say
that $(i,j),(i',j')\in\kp(T,K,L)$ are \emph{equivalent up to $G_L$ and $G_T$} if there
is an isometry $f\in G_L$ such that $f(iT\oplus jK)=i'T\oplus j'K$ and $f(iT)=i'T$ and
$f_T\in G_T$ for the induced isometry.

For later use, we now present a version of Lemma~\ref{lem:G-eq1} in the presence of
a subgroup $G_L$ but with Assumption~\ref{assumption}.

\begin{lemma} \label{lem:G-eq2}
Assume that $L$ is an overlattice of $T\oplus K$ which is unique within its genus.
Let $H,H'\in\kq(T,K,q_L)$.

Then $\mapL(H)$ and $\mapL(H')$ are equivalent in $\kp(T,K,L)$ up to $G_L$ and
$G_T$ if and only if there is an isometry $\psi\in G_T\times\OO(K)$ such that
$D_\psi(H)=H'$ and $\psi^\vee|_L\in G_L$.
\end{lemma}

\begin{proof} \ 
Note that the $G_L$-action is well defined by Assumption~\ref{assumption}.
The proof of the lemma is the same as the one of Lemma~\ref{lem:G-eq1},
taking into account the additional assumption.
\qed \end{proof}

\begin{lemma}[{\cite[Lemma 23]{morrison}}] \label{lem:morrison}
Let $L$ be an overlattice of $T\oplus K$ such that $L$ is unique in its
genus and let $K'$ be a lattice in the genus of $K$. Then there is a 
bijection $\ko(T,K,L)\bijection\ko(T,K',L)$.

In particular, there is a primitive embedding $K'\embed L$ such that
$L$ becomes an overlattice of $T\oplus K'$.
\end{lemma}

\begin{proof} \ 
We observe that the set $\kq(T,K,q)=\kq(T,K,L)$ does not really depend on $K$, but
rather just on the discriminant form $q_K$.
Hence from Lemma~\ref{lem:nikulin151} and using Assumption~\ref{assumption}
we get a chain of bijections
\[ \ko(T,K,L) \bijection \kq(T,K,L) \bijection \kq(T,K',L) \bijection \ko(T,K',L) \]
and hence the claim.
\qed\end{proof}

The situation is particularly nice for indefinite unimodular overlattices
where we recover a result proved by Hosono et al.:

\begin{corollary}[{\cite[Theorem 1.4]{HLOY2002}}]  \label{cor:unpolarised-case}
Let $T\oplus K$ be indefinite. Then there is a bijection
 $G_T\backslash\kq(T,K,0) \bijection  G_T\times\OO(K)\backslash\OO(D_K)$,
where $G_T$ acts on $D_K$ via $G_T\embed\OO(T)\to\OO(D_T)\isom\OO(D_K)$.
\end{corollary}

\begin{proof} \ 
We have $D_T\cong D_K$ by the following standard argument: the map
$L=L^\vee\to T^\vee$ is surjective with kernel $K$, hence $L\cong T^\vee\oplus K$,
and $T^\vee/T\cong(T^\vee\oplus K)/(T\oplus K)\cong L/(T\oplus K)$;
similarly for $K^\vee/K$ by symmetry. Also, the forms on $D_T$ and $D_K$
coincide up to sign: $q_T\cong-q_K$. This also shows that subgroups $H$
of Lemmas~\ref{lem:G-eq1} and \ref{lem:G-eq2} are graphs of isomorphisms.

Therefore, primitive embeddings $T\embed L$ are determined by anti-isometries
$\gamma\colon D_T\isom D_K$. If there exists such an embedding (and hence such
an anti-isometry), this set is bijective to $\OO(D_T)$. We deduce the claim
from Lemma \ref{lem:G-eq2}.
\qed \end{proof}

\begin{remark} \label{rem:unimodular-overlattice}
Note that in the unimodular case ($q=0$), the prescription of $T$ and of the
genus of the overlattice (i.e.\ just the signature in this case) already
settles the genus of $K$ by $q_K=-q_T$ and the signature of $K$ is
obviously fixed. This statement is wrong in the non-unimodular case:
It can happen that a sublattice has two embeddings with orthogonal
complements of different genus as in the following example.
\end{remark}

\begin{example}
Let $T:=\langle2\rangle$ with generator $t$ and $L:= U\oplus\langle2\rangle$ with
generators $e,f\in U$, $x\in\langle2\rangle$. Consider the embeddings
$\iota_1,\iota_2\colon T\embed L$ given by $\iota_1(t)=e+f$ and $\iota_2(t)=x$.
Then, bases for the orthogonal complements are $\{e-f,x\}\subset\iota_1(T)\orth$
and $\{e,f\}\subset\iota_2(T)\orth$. Hence $\text{disc}(\iota_1(T)\orth)=4$ but
$\text{disc}(\iota_1(T)\orth)=1$.
\end{example}



\section{$\Kthree$ surfaces} \label{sec:K3}

\noindent
In this text, a \emph{$\Kthree$ surface} will mean a smooth compact complex surface which is
simply connected and carries a no-where vanishing 2-form. The latter two conditions
are equivalent to zero irregularity ($H^1(X,\ko_X)=0$) and trivial canonical bundle
($\Omega_X^2\cong\ko_X$). See \cite[VIII]{bpvdv} or \cite{beauville-k3} for details.


We denote the \emph{Picard rank} of a $\Kthree$ surface $X$ by $\rho_X$. It is
the number of independent divisor classes or, equivalently, of independent
line bundles on $X$. If $X$ is projective, then $\rho_X\geq1$ but not vice versa.
The cohomology groups listed below carry lattice structures coming from the
cup product on the second cohomology:
\begin{align*}
 H^2_X &= H^2(X,\IZ)  && \text{ full second cohomology}, && \sign(H^2_X)=(3,19) \\
 T_X &                && \text{ transcendental lattice}, && \sign(T_X)=(2,20-\rho_X)\\
 NS_X &               && \text{ N\'eron-Severi lattice}, && \sign(NS_X)=(1,\rho_X-1)
\end{align*}
where the signatures in the second and third cases are valid only for
$X$ projective. Following usage in algebraic geometry, we will often write
$\alpha.\beta=(\alpha,\beta)$ for the pairing. Likewise, we will use the
familiar shorthand $L.M$ for the pairing of the first Chern classes $c_1(L).c_1(M)$
of two line bundles $L$ and $M$.

By Poincar\'e duality, $H^2_X$ is a unimodular lattice; it follows from Wu's formula
that the pairing is even.
Indefinite, even, unimodular lattices are uniquely determined by their
signature; we get that $H^2_X$ is isomorphic to the so-called \emph{$\Kthree$ lattice}
made up from three copies of the hyperbolic plane $U$ and two copies of the negative
$E_8$ lattice:
\begin{align*}
L_\kd &= 3U\oplus2E_8(-1) .
\end{align*}

The N\'eron-Severi and transcendental lattices are mutually orthogonal primitive
sublattices of $H^2_X$. In particular, $H^2_X$ is an overlattice of $T_X\oplus NS_X$.

We denote by $\omega_X$ the canonical form on $X$. It has type $(2,0)$ and is
unique up to scalars, since $H^0(X,\Omega_X^2)=\IC$ for a $\Kthree$ surface. By abuse
of notation, we also write $\omega_X$ for its cohomology class, so that
$\omega_X\in T_X\otimes\IC$. In fact, $T_X$ is the smallest primitive sublattice
of $H^2_X$ whose complexification contains $\omega_X$.

As $X$ is a complex K\"ahler manifold, the second cohomology $H^2_X$
comes equipped with a pure Hodge structure of weight 2:
 $H^2_X\otimes\IC = H^{2,0}(X) \oplus H^{1,1}(X) \oplus H^{0,2}(X)$.
Note that $H^{1,1}(X) = (\IC\omega_X+\IC\overline\omega_X)\orth$.
The transcendental lattice $T_X$ is a Hodge substructure with unchanged
$(2,0)$ and $(0,2)$ components.

A \emph{Hodge isometry} of $H^2_X$ (or $T_X$) is an isometry that maps
each Hodge summand to itself. As the $(2,0)$-component is one-dimensional,
Hodge isometries are just isometries $\varphi\colon H^2_X\isom H^2_X$ with
$\varphi_\IC(\omega_X)=c\omega_X$ for some $c\in\IC^*$. (Analogous for
Hodge isometries of $T_X$.) If $L$ is a lattice with Hodge structure, we
denote the group of Hodge isometries by $\OO_\hodge(L)$.

The following two Torelli theorems are basic for all subsequent work. They say
that essentially everything about a $\Kthree$ surface is encoded in its second
cohomology group, considered as a lattice with Hodge structure --- for both
the classical and derived point of view. (We repeat Orlov's result about
equivalent surfaces up to derived equivalence.)

\begin{classicalTorelli}
Two $\Kthree$ surfaces $X$ and $Y$ are isomorphic if and only if there is a Hodge
isometry between their second cohomology lattices $H^2_X$ and $H^2_Y$.
\end{classicalTorelli}

\begin{derivedTorelli}[Orlov]
Two projective $\Kthree$ surfaces $X$ and $Y$ are derived equivalent if and only if
there is a Hodge isometry between the transcendental lattices $T_X$ and $T_Y$.
\end{derivedTorelli}


See \cite{beauville-k3} or \cite[\S VIII]{bpvdv} for the classical case (the
latter reference gives an account of the lengthy history of this result), and
\cite{orlov-k3} or \cite[\S10.2]{huybrechts} for the derived version.

%

A \emph{marking} of $X$ is the choice of an isometry $\lambda_X\colon H^2_X\isom L_\kd$.
The period domain for $\Kthree$ surfaces is the following open subset of the projectivised $\Kthree$ lattice:
\[ \Omega_{L_\kd} = \{ \omega\in\IP(L_\kd\otimes\IC) \mid \omega.\omega=0, \omega.\overline\omega>0 \} . \]
Since $L_\kd$ has signature $(3,k)$ with $k>2$, this set is connected.
By the surjectivity of the period map \cite[VIII.14]{bpvdv}, each point of $\Omega_{L_\kd}$
is obtained from a marked $\Kthree$ surface. Forgetting the choice of marking by dividing out the
isometries of the $\Kthree$ lattice, we obtain a space $\kf=\OO(L_\kd)\backslash\Omega_{L_\kd}$
parametrising all (unmarked) $\Kthree$ surfaces. As is well known, $\kf$ is a 20-dimensional,
non-Hausdorff space. In particular, it is not a moduli space in the algebro-geometric sense.

Denote by $\kd_\FM$ the set of all $\Kthree$ surfaces up to derived equivalence --- two
$\Kthree$ surfaces get identified if and only if they are Fourier-Mukai partners, i.e.\ if
and only if their transcendental lattices are Hodge-isometric. Its elements are the
sets $\FM(X)$ of Fourier-Mukai partners of $\Kthree$ surfaces $X$.
One cannot expect this set to have a good analytic structure: the fibres of the map
$\kf\to\kd_\FM$ can become arbitrarily large (see \cite{oguiso-primes}).
On the other hand, any $\Kthree$ surface has only finitely many FM partners
(\cite{bridgeland-maciocia}), so that the fibres are finite at least.

Since the transcendental lattices determine D-equivalence by Orlov's derived
Torelli theorem, the Fourier-Mukai partners of a $\Kthree$ surface $X$ are given by
embeddings $T_X\subseteq L_\kd$, modulo automorphisms of $T_X$. This can be
turned into a precise count:

\begin{thm-non}[Hosono, Lian, Oguiso, Yau {\cite[Theorem 2.3]{HLOY2002}}]
The set of Fourier-Mukai partners of a $\Kthree$ surface $X$ has the following partition
\[ \FM(X) = \coprod_{S\in\kg(NS_X)} \OO_\hodge(T_X) \times \OO(S) \backslash \OO(D_S) \]
with $\OO(S)$ and $\OO_\hodge(T_X)$ acting on $\OO(D_S)$ as in Corollary~\ref{cor:unpolarised-case}
above.
\end{thm-non}


The special case of a generic projective $\Kthree$ surface, $\rk(NS_X)=1$, was treated
before, leading to a remarkable formula reminiscent of classical genus theory for
quadratic number fields (and proved along these lines):

\begin{thm-non}[Oguiso {\cite{oguiso-primes}}]
Let $X$ be a projective $\Kthree$ surface with $\mathrm{Pic}(X)$ generated by an ample line bundle
of self-intersection $2d$. Then $X$ has $2^{p(d)-1}$ FM partners, where $p(d)$ is the number
of distinct prime factors of $d$, and $p(1)=1$.
\end{thm-non}

Oguiso's theorem can also be interpreted as a result about polarised $\Kthree$ surfaces,
which we turn to next. In particular, the number $2^{p(d)-1}$ is the order of 
$\OO(D_{L_{2d}})/\langle\pm1\rangle$, where$L_{2d}$ is the replacement of the lattice
$L_\kd$ in the polarised case.


\section{Polarised $\Kthree$ surfaces}\label{sec:polarisedK3surfaces}

\noindent
A \emph{semi-polarised $\Kthree$ surface} of degree $d>0$ is a pair $(X,h_X)$
of a $\Kthree$ surface $X$ together with a class $h_X\in NS_X$ of a nef divisor
with $h_X^2=2d>0$. A nef divisor of positive degree is also called
\emph{pseudo-ample}. We recall that an effective divisor is nef if and only if
it intersects all $-2$-curves nonnegatively \cite[\S VIII.3]{bpvdv}.
We will also assume that $h_X$ is primitive, i.e.\ not a non-trivial integer 
multiple of another class.

We speak of a \emph{polarised $\Kthree$ surface} $(X,h_X)$ if $h_X$ is the class
of an ample divisor. However, we call $h_X$ the \emph{polarisation}, even
if it is just nef and not necessarily ample. For details, see \cite[\S VIII.22]{bpvdv}.
The relevant geometric lattice is the complement of the polarisation
\begin{align*}
 H_X &= (h_X)\orth_{H^2_X}                         && \text{non-unimodular of signature }(2,19).
\intertext{%
which inherits lattice and Hodge structures from $H^2_X$. \newline
On the side of abstract lattices, recall that $L_\kd=3U\oplus2E_8(-1)$;
we denote the three orthogonal copies of $U$ in $L_\kd$ by $U^{(1)}$, $U^{(2)}$, and $U^{(3)}$.
Basis vectors $e_i$, $f_i$ of $U^{(i)}$, defined by $e_i^2=f_i^2=0$ and $e_i.f_i=1$, always refer
to such a choice. For $h\in L_\kd$ with $h^2>0$, set
}
L_h    &= h\orth_{L_\kd}                            && \text{non-unimodular of signature } (2,19), \\
L_{2d} &= 2U\oplus\langle-2d\rangle\oplus2E_8(-1)  && \text{the special case } h=e_3+df_3.
\end{align*}

Since all primitive vectors of fixed length appear in a single $\OO(L_\kd)$-orbit
by Eichler's criterion, we can assume  $h=e_3+df_3$. Note that $H_X\cong L_{2d}$ as
lattices. Obviously, $D_{L_{2d}}$ is the cyclic group of order $2d$. The non-unimodular
summand $\langle-2d\rangle$ of $L_{2d}$ is generated by $e_3-df_3$; thus $D_{L_{2d}}$ is
generated by the integer-valued functional $\frac{1}{2d}(e_3-df_3,\cdot)$. The quadratic
form $D_{L_{2d}}\to\IQ/2\IZ$ is then given by mapping this generator to the class of
$\frac{-2d}{4d^2}=\frac{-1}{2d}$.

There are two relevant groups in this situation: the full isometry group $\OO(L_{2d})$
and the subgroup $\tilde\OO(L_{2d})$ of \emph{stable} isometries which by definition
act trivially on the discriminant $D_{L_{2d}}$. The next lemma gives another description
of stable isometries.

\begin{lemma}
The stable isometry group coincides with the group of $L_\kd$-isometries stabilising
$h$, i.e.\ $\tilde\OO(L_{2d})=\{g\in\OO(L_\kd)\mid g(h)=h\}$.
\end{lemma}

\begin{proof} \ 
Given $g\in\OO(L_\kd)$ with $g(h)=h$, we make use of the fact that the discriminant 
groups of $h\orth=L_{2d}$ and $\langle2d\rangle$ (the latter generated by $h$) are
isomorphic and their quadratic forms differ by a sign. This is true because these
are complementary lattices in the unimodular $L_\Kthree$; see the proof of
Corollary~\ref{cor:unpolarised-case}. The induced maps on discriminants, 
 $D_{g,h\orth}\colon h\orth\isom h\orth$ and 
 $D_{g,h}\colon\langle2d\rangle\isom\langle2d\rangle$,
are the same under the above identification. Since $D_{g,h}$ is the identity by 
assumption, $D_{g,h\orth}$ is, too. Hence, $g|_{L_{2d}}$ is stable.

On the other hand, any $f\in\OO(L_{2d})$ allows defining an isometry $\tilde f$
of the lattice $L_{2d}\oplus\IZ h$, by mapping $h$ to itself. Note that
$L_{2d}\oplus\IZ h\subset L_\kd$ is an overlattice. If $f$ is a stable isometry,
i.e.\ $D_f=\id$, then $\tilde f$ extends to an isometry of $L_\kd$. (This
reasoning has already been alluded to in Remark~\ref{rem:isometry-extension}.)
%
%
\qed \end{proof}

An isomorphism of semi-polarised $\Kthree$ surfaces is an isomorphism of the surfaces
respecting the polarisations. Here we recall two Torelli theorems
which are essential for the construction of moduli spaces of $\Kthree$ surfaces.

\begin{StrongPolarisedTorelli}
Given two properly polarised $\Kthree$ surfaces $(X,h_X)$ and $(Y,h_Y)$, i.e.\ $h_X$ and
$h_Y$ are ample classes, and a Hodge isometry $\varphi\colon H^2_X\isom H^2_Y$ with
$\varphi(h_X)=h_Y$, there is an isomorphism $f\colon Y\isom X$ such that $\varphi=f^*$.
\end{StrongPolarisedTorelli}

This result only holds for polarised $\Kthree$ surfaces. For semi-polarised $\Kthree$ surfaces we have
a different result, namely
\begin{SemiPolarisedTorelli} $\quad$
Two semi-polarised $\Kthree$ surfaces $(X,h_X)$ and $(Y,h_Y)$ are isomorphic if and only
if there is a Hodge isometry $\varphi\colon H^2_X\isom H^2_Y$ with $\varphi(h_X)=h_Y$.
\end{SemiPolarisedTorelli}
\begin{proof} \  Let $\varphi\colon H^2_X\isom H^2_Y$ be a Hodge isometry with $\varphi(h_X)=h_Y$.
Since $h_X$ is not ample we cannot immediately invoke the strong Torelli theorem. Following
\cite[p.151]{beauville-k3}, let $\Gamma$ be the subgroup of the Weyl group of $Y$ generated
by those roots of $H^2(Y,\mathbb Z)$ which are orthogonal to $h_Y$. Then $\Gamma$ acts
transitively on the chambers of the positive cone. Hence we can find an element $w\in\Gamma$
such that $w(h_Y)=h_Y$ and $w\circ\varphi$ maps the ample cone of $X$ to the ample cone of $Y$.
By the strong Torelli theorem  $w\circ\varphi$ is now induced by an isomorphism
$f\colon Y \isom X$. This gives the claim.
\qed \end{proof}


The counterpart of the unpolarised period domain is the open subset
\[ \Omega_{L_{2d}}^\pm = \{ \omega\in\IP(L_{2d}\otimes\IC)=\IP^{20} \mid
                          (\omega,\omega)=0, (\omega,\overline\omega)>0 \}
                    =  \Omega_{L_\kd}\cap h^\perp \]
where we abuse notation to also write $h^\perp$ for the projectivised hyperplane.
Obviously, both $\OO(L_{2d})$ and its subgroup $\tilde\OO(L_{2d})$ act on
$\Omega_{L_{2d}}^\pm$. Since the signature of $L_{2d}$ is $(2,19)$, the action is
properly discontinuous.

Furthermore, signature $(2,19)$ also implies that $\Omega_{L_{2d}}^\pm$ has two
connected components. These are interchanged by the (stable) involution induced by
 $\id_{U^{(1)}}\oplus(-\id_{U^{(2)}})\oplus\id_{U^{(3)}\oplus2E_8(-1)}$. Denote by
$\Omega_{L_{2d}}^+$ one connected component; this is a type IV domain. Also, let
$\OO_{L_{2d}}^+$ and $\tilde\OO_{L_{2d}}^+$ be the subgroups of the (stable) isometry
group of $L_{2d}$ fixing the component. They are both arithmetic groups, as they have
finite index in $\OO(L_{2d})$.

Next, let $\Delta\subset L_{2d}$ be the subset of all $(-2)$-classes, and for $\delta\in\Delta$
denote by $\delta\orth\subset L_{2d}\otimes\IC$ the associated hyperplane (`wall'). In analogy
to the unpolarised case we define a parameter space as the quotient by the group action ---
however, there are certain differences to be explained below: let
\begin{align*}
 \kf_{2d}      &=  \tilde\OO(L_{2d}) \backslash \Omega_{L_{2d}}^\pm . \\
\intertext{This space has an analytic structure as the quotient of a type IV domain by a group
           acting properly discontinuously. Furthermore, $\kf_{2d}$ is actually quasi-projective
           by Baily-Borel \cite{bb}. Note that the group actions preserve the collection of walls
           $\delta\orth$, which by abuse of notation are given the same symbol in the quotient.
           Hence, the group action also preserves the complement}
 \kf_{2d}^\circ &= \kf_{2d} \setminus \bigcup_{\delta\in\Delta} \delta\orth .
\end{align*}

The condition on $\kf_{2d}^\circ$ means that $-2$-classes orthogonal to the polarisation are
transcendental. In other words, the polarisation is ample, as it is nef and non-zero on all
$-2$-curves.

The subspace $\kf_{2d}^\circ$ is the moduli space of pairs $(X,h_X)$ consisting of
(isomorphism classes of) a $\Kthree$ surface $X$ and the class $h_X$ of an ample, primitive
line bundle with $h_X^2=2d$: given such a pair, we choose a \emph{marking}, i.e.\ an isometry
$\lambda_X\colon H_X^2\isom L_{\kd}$ such that $\lambda_X(h_X)=h$. This induces
$\lambda_X|_{H_X}\colon H_X\isom L_{2d}$ and gives the period point
$\lambda_X(\omega_X)\in\Omega_{2d}^\pm$. Since $h_X$ is ample, the period point avoids the walls.

Conversely, given an $\tilde\OO(L_{2d})$-orbit of a point $[\omega]\in\Omega_{2d}^\pm$ not
on any wall, we get a pair $(X,h_X)$ by considering $[\omega]$ as period point for the full $\Kthree$
lattice: this uses the surjectivity of the period map. Now our assumptions on $\omega$ imply
$h_X^2=2d$ and that $h_X$ is ample as $\omega$ avoids the walls. Then, the strong Torelli theorem
says that both the $\Kthree$ surface $X$ and the polarisation $h_X$ are unique (up to isomorphism).

Finally, using again the surjectivity of the period map, one can find for every element
$[\omega]\in\kf_{2d}\setminus\kf_{2d}^\circ$ a semi-polarised $\Kthree$ surface $(X,h_X)$
and a marking $\varphi\colon H^2(X,\mathbb X) \to  L_\kd $ with $\varphi(h_X)=h$,
$\varphi([\omega_X])=\omega$. The fact that the points not contained in $\kf_{2d}^\circ$
correspond to isomorphism classes of semi-polarised $\Kthree$ surfaces of degree $2d$ now
follows from the Torelli theorem for semi-polarised $\Kthree$ surfaces.



\begin{example}
For $d=1$, the smallest example of a proper semi-polarisation (i.e.\ nef, not ample)
occurs for a generic elliptic $\Kthree$ surface $X$ with section. Its N\'eron-Severi lattice will
be generated by the section $s$ and a fibre $f$. The intersection form on $NS(X)$ is
$\sqmat{-2}{1}{1}{0}$ and we set $D:=s+2f$. This effective divisor is primitive and nef
as $D.s=0$, $D.f=1$, $D^2=2$.
\end{example}

We remark that the lattice $L_{2d}$ is more difficult to work with than $L_\kd$
as it is not unimodular anymore. On the other hand, the moduli space $\kf_{2d}$
of $2d$-polarised $\Kthree$ surfaces is a quasi-projective variety which is a huge
improvement over $\kf=\OO(L_{\kd})\backslash\Omega_{L_{\kd}}$.

In the polarised case, another natural quotient appears, taking the full
isometry group of the lattice $L_{2d}$:
\[ \hat\kf_{2d} = \OO(L_{2d}) \backslash \Omega_{2d}^\pm  . \]
(The unpolarised setting has $D_{L_\kd}=0$ and hence there is only one natural
group to quotient by.)

There is an immediate quotient map $\pi\colon\kf_{2d}\to\hat\kf_{2d}$. It has finite
fibres and was investigated by Stellari:

\begin{lemma}[{\cite[Lemma~2.3]{stellari}}] The degree of $\pi$ is $\deg(\pi)=2^{p(d)-1}$
where $p(d)$ is the number of distinct primes dividing $d$.
\end{lemma}

\begin{proof} \ 
The degree is given by the index $[\OO(L_{2d}):\tilde\OO(L_{2d})]$ up to the action
of the non-stable isometry $-\id$ which permutes the two components.

We use the exact sequence  $0\to\tilde\OO(L_{2d})\to\OO(L_{2d})\to\OO(D_{L_{2d}})\to0$
where the right-hand zero follows from \cite[Theorem~1.14.2]{nikulin}. The index
thus equals the order of the finite group $\OO(D_{L_{2d}})$. Now $D_{L_{2d}}$ is the
cyclic group of order $2d$ and decomposes into the product of various $p$-groups.
Automorphisms of $D_{L_{2d}}$ factorise into automorphisms of the $p$-groups. However,
the only automorphisms of $\IZ/p^l$ respecting the quadratic (discriminant) form
are those induced by $1\mapsto\pm1$. Hence $|\OO(D_{L_{2d}})|=2^{p(d)}$. The degree is
then $2^{p(d)-1}$, taking the non-stable isometry $-\id$ of $L_{2d}$ into account.

In case $d=1$, we have $\tilde\OO(L_{2d})=\OO(L_{2d})$, fitting with $p(1)=1$.
\qed \end{proof}

Points of $\hat\kf_{2d}$ or rather the fibres of $\pi$ have the following property:

\begin{lemma} \label{lem:tau-fiber}
Given semi-polarised $\Kthree$ surfaces $(X,h_X)$ and $(Y,h_Y)$ with
$\pi(X,h_X)=\pi(Y,h_Y)$, there are Hodge isometries $H_X\cong H_Y$ and
$T_X\cong T_Y$. In particular, $X$ and $Y$ are FM partners.
\end{lemma}

\begin{proof} \ 
Fix markings $\lambda_X\colon H^2_X\isom L_\kd$ and $\lambda_Y\colon H^2_Y\isom L_\kd$
with $\lambda_X(h_X)=\lambda_Y(h_Y)=h$ and $\lambda_X(\omega_X)=\omega$,
$\lambda_Y(\omega_Y)=\omega'$.
By $\pi(X,h_X)=\pi(Y,h_Y)$, there is some $g\in\OO(L_{2d})$ such that
$g(\lambda_X(\omega_X))=\lambda_Y(\omega_Y)$. In particular, the primitive lattices
generated by $\omega_X$ and $\omega_Y$ (which are the transcendental lattices $T_X$
and $T_Y$) get mapped into each other by $\lambda_Y\inv\circ g\circ\lambda_X$. Thus,
the latter isometry respects the Hodge structures, and induces Hodge isometries
$T_X\isom T_Y$ and $H_X=h_X\orth\isom H_Y=h_Y\orth$.
\qed \end{proof}


\section{Polarisation of FM partners}

\noindent
In this section, we want to consider the relationship between polarisations and FM partners.
A priori these concepts are very different: the condition that two K3 surfaces are
derived equivalent is a property of their transcendental lattices, whereas the existence of
polarisations concerns the N\'eron-Severi group. Indeed, we shall see that there are 
FM partners where one K3 surface carries a polarisation of given degree but the other does 
not. On the opposite side, we shall see in the next Section~\ref{sec:countingformula} that
one can count the number of FM partners among polarised K3 surfaces of a given degree.

Introduce the set
\[ \kd_\FM^{2d} := \Big\{
                   \begin{array}{@{}l@{}l@{}} X & \text{ $\Kthree$ surface admitting a primitive } \\
                                                & \text{ nef line bundle } L \text{ with } L^2=2d
                   \end{array}
                   \Big\} / \sim \]
where $X\sim Y$ if and only if $D^b(X)\cong D^b(Y)$.

We shall first discuss two examples which shed light on the relationship between FM partnership
and existence of polarisations.

\begin{example}
Derived equivalence does not respect the existence of polarisations of a given degree. 
To give an example we use the rank 2 N\'eron-Severi lattices defined by 
$\sqmat{0}{-7}{-7}{-2}$ and $\sqmat{0}{-7}{-7}{10}$ which are related by the rational
base change $\frac{1}{3}\sqmat{1}{0}{-2}{9}$. The former obviously represents $-2$ 
whereas the latter does not. Furthermore, the latter primitively represents 10 via the
vector $(0,1)$ and 6 via the vector $(2,3)$ whereas the former does not. For example, 
if we had $10=-14xy-2y^2$, then $y$ would have to be one of $1, 2, 5, 10$ up to sign 
and neither of these eight cases works.

The orthogonal lattices in $L_\Kthree$ are isomorphic, as follows from Nikulin's criterion.
Denote the common orthogonal complement by $T$. As in the previous example, we choose a 
general vector $\omega\in T_\IC$ with $(\omega,\omega)=0$, $(\omega,\overline\omega)>0$.
We see that $T$ admits primitive embeddings $\iota,\iota'\colon T\embed L_\Kthree$ such
that $\iota(T)^\perp$ does not contain any vectors of square $2d$ whereas $\iota'(T)^\perp$
does, for $d=3$ or $d=5$. Furthermore, $\iota'(T)^\perp$ does not contain any $-2$-classes.
The surface $X'$ corresponding to $\iota'(\omega)\in T_\IC$ can actually be $2d$-polarised 
since $\iota'(\omega)\in L_h^\perp\cong L_{2d}$ and we have $\kf_{2d}^\circ = \kf_{2d}$ by the
absence of $-2$-classes.
On the other hand, the surface $X$ corresponding to $\iota(\omega)$ has no 
$2d$-(semi)polarisations --- there are not even classes in $NS(X)$ of these degrees.

\end{example}

\begin{example} 
Let $d >1$ be an integer, not divisible by $3$ such that $2d$ can be represented primitively by $A_2$ (e.g. $d=7$).
Let $T=2U\oplus E_8(-1) \oplus E_6(-1) \oplus \langle-2d\rangle$. Following an idea
of M.~Sch\"utt we construct two primitive embeddings $\iota,\iota'\colon T\embed L_{2d}$. Both of them are the identity on $2U\oplus E_8(-1)$. On the $E_6(-1)\oplus\langle-2d\rangle$ part of $T$, we use
\[ \begin{array}{rll}
\iota\colon & E_6(-1) \embed E_8(-1), \langle-2d\rangle \isom \langle-2d\rangle & \text{ with } \iota(T)_{L_{2d}}^\perp \cong A_2(-1) \\
\iota'\colon & E_6(-1)\oplus\langle-2d\rangle \embed E_8(-1) & \text{ with } \iota'(T)_{L_{2d}}^\perp \cong \langle-6d\rangle\oplus\langle-2d\rangle.
\end{array} \]

We choose a general point $\omega \in T_\IC$ with $(\omega,\omega)=0$ and $(\omega,\bar{\omega})>0$.
Then the points $\iota(\omega)$ and $\iota'(\omega)$ in $\kf_{2d}$ represent (semi-)polarised $\Kthree$ surfaces
$(X,h_X)$ and $(X',h_{X'})$. In the first case $h_X$ is only semi-polarised, as $A_2(-1)$ contains $(-2)$-vectors,
in other words $\iota(\omega) \in \kf_{2d} \setminus \kf_{2d}^\circ $. In the second case 
$\iota(\omega) \in \kf_{2d}^\circ $ since the orthogonal complement of $h_{X'}$ in $NS(X')$ equals 
$\langle-2d\rangle\oplus\langle-6d\rangle$ which does not contain a $(-2)$-class. This shows that there are examples
of polarised and semi-polarised $\Kthree$ surfaces of the same degree which have the same FM partner. 

Incidentally we notice that $NS(X) \cong NS(X') \cong A_2(-1) \oplus \langle 2d \rangle$. This follows 
from Nikulin's criterion since both lattices have rank $3$ and length $1$ since we have 
assumed that $(3,d)=1$.  



\end{example}

We want to construct a map
\[ \tau\colon \kf_{2d} \to \kd_\FM^{2d} . \]
By Lemma~\ref{lem:tau-fiber}, we have a map $\pi\colon \hat\kf_{2d} \to \kd_\FM^{2d}$;
combining it with $\sigma\colon \kf_{2d} \to \hat\kf_{2d}$, we obtain a commutative
triangle


\[ \xymatrix{
\kf_{2d} \ar[rr]^\pi \ar[rdd]_\tau & &
\hat\kf_{2d} \ar[ddl]^\sigma \\ \\ &
\kd^{2d}_\FM
} \]

By the counting results of Proposition~\ref{FM-fiber}, the fibres are finite. Here we
give a geometric argument for that fact, following Stellari \cite[Lemma 2.5]{stellari-group},
where we pay special attention to the `boundary points' of $\kf_{2d}$.

\begin{proposition}
Given a $2d$-(semi-)polarised $\Kthree$ surface $(X,h_X)$, there are only finitely many
$2d$-(semi-)polarised $\Kthree$ surfaces $(Y,h_Y)$ up to isomorphism with $D^b(X)\cong D^b(Y)$.
\end{proposition}

\begin{proof} \ 
Disregarding polarisations, there are only finitely many FM partners of $X$, as $X$
is a smooth projective surface \cite{bridgeland-maciocia}. Given such an FM partner
$Y$, consider the set $A_{Y,2d}=\{c\in\kc(Y) \mid c^2=2d\}$ of elements of length $2d$
in the positive cone of $Y$. If the divisor $c$ is ample, then $3c$ is very ample,
by Saint-Donat's result \cite{saint-donat}. By Bertini, there are irreducible
divisors $D\in|3c|$.
The set  $B_{Y,18d}=\{\ko_X(D)\mid D^2=18d, D\text{ irreducible}\}$ of divisor
classes of irreducible divisors of length $18d$ is finite up to automorphisms by
Sterk \cite{sterk}.
As $A_{Y,2d}\to B_{Y,18d}$, $c\mapsto|3c|$ is injective (this uses $H^1(\ko_X)=0$),
this shows that the number of non-isomorphic $2d$-polarisations on $Y$ is finite.

However, there are points $(Y,h_Y)$ where $h_Y$ is only pseudo-ample. Denote the set
of pseudo-ample divisors of degree $2d$ by
 $\overline A_{Y,2d}=\{c\in\overline\kc(Y) \mid c^2=2d\}$. If we have a non-ample
polarisation $h_Y$, then contracting the finitely many $-2$-curves which intersect to
zero with $h_Y$ produces a projective surface $Y'$ with only ADE singularities,
trivial canonical bundle and $H^1(\ko_{Y'})=0$. Morrison shows that Saint-Donat's
result is also true for this surface \cite[\S6.1]{morrison-k3}, i.e.\ $3c$ is again
very ample. We can then proceed as above, as the generic divisor in $|3c|$ will be
irreducible and avoid the finitely many singularities. Sterk's result on finiteness
of $B_{Y',18d}/\Aut(Y')$ still applies as he simply assumes that a linear system is
given whose generic member is irreducible.
\qed \end{proof}


\section{Counting FM partners of polarised $\Kthree$ surfaces in lattice terms}\label{sec:countingformula}

\noindent
Taking our cue from the fact that the fibres of $\kf\to\kf/\FM$ are just given by
FM partners (the unpolarised case), and the latter can be counted in lattice terms,
we study the following general setup: let $L$ be an indefinite, even lattice, let
$T$ be another lattice, occurring as a sublattice of $L$, and let finally
$G_T\subseteq\OO(T)$ and $G_L\subseteq\OO(L)$ be two subgroups, the latter normal.
As in Section~\ref{sec:overlattices}, we consider the set $\kp(T,L)$ of all
primitive embeddings $\iota\colon T\embed L$. This set is partitioned into
$\kp(T,K,L)$, containing all primitive embeddings $\iota\colon T\embed L$ with
$\iota(T)\orth_L\cong K$.

In the application to geometry, we will have $L=L_{2d}=h^\perp$ the perpendicular
of the polarisation inside the $\Kthree$ lattice, $T=T_X$ the transcendental lattice of
a $\Kthree$ surface $X$ and $K=NS(X)$ the N\'eron-Severi lattice of $X$. By Nikulin's
criterion, $L_{2d}$ is unique in its genus, thus fulfilling Assumption~\ref{assumption}.
As to the groups, $G_T=\OO_\hodge(T)$ is the group of Hodge isometries of $T_X$ and
$G_L$ is either the full or the stable isometry group of $L_{2d}$.

We recall when two embeddings $\iota_1, \iota_2\colon T\embed L$ are equivalent
with respect to $G_T$ and $G_L$ (see page~\pageref{assumption}): if there are
isometries $g\in G_T$ and $\tilde g\in G_L$ such that
\[ \xymatrix{
T \ar^{\iota_1}[r] \ar[d]^g & L \ar[d]^{\tilde g} \\
T \ar^{\iota_2}[r]          & L
} \]
This corresponds to orbits of the action $G_T\times G_L\times\kp(T,L)\to\kp(T,L)$,
$(g,\tilde g)\cdot\iota=\tilde g\iota g\inv$.

All of this is essentially the setting of \cite{HLOY2002} --- the novelty is
the subgroup $G_L$, which always was the full orthogonal group in loc.\ cit.

\begin{proposition} \label{FM-fiber}
For a $2d$-polarised $\Kthree$ surface $(X,h_X)$, there are bijections
\begin{align*}
\sigma\inv([X,h_X]) &\bijection \OO_\hodge(T_X)\times\OO(H_X) \backslash \kp(T_X,H_X), \\
\tau\inv([X,h_X])   &\bijection \OO_\hodge(T_X)\times\tilde\OO(H_X) \backslash \kp(T_X,H_X).
\end{align*}
\end{proposition}

\begin{remark} \label{rem:unpolarized}
The unpolarised analogue of the proposition was given in Theorem 2.4 of
\cite{HLOY2002}, stating
 $\FM(X) = \OO_\hodge(T_X)\times\OO(H^2_X) \backslash \kp(T_X,H^2_X)$.
\end{remark}

\begin{proof} \ 
The proof proceeds along the lines of \cite[Theorem 2.4]{HLOY2002}.
Fix a marking $\lambda_X \colon H_X\isom L_{2d}$ for $X$. Set $T:=\lambda_X(T_X)$.
This yields a primitive embedding
\[ \iota_0 \colon \xymatrix{
      T\ar[r]^-\sim_-{\lambda_X\inv} & T_X \ar@{^{(}-}[r] &
      H_X \ar[r]^\sim_{\lambda_X} & L_{2d}
} . \]
This embedding (or rather the equivalence class of $\iota_0\lambda_X(\omega_X)$)
gives a point in $\kf_{2d}$. By definition of $\kf_{2d}$, this period point does
not depend on the choice of marking.

If $(Y,h_Y)$ belongs to (a period point given by an embedding in) $\kp(T,L_{2d})$,
then --- as the transcendental lattice is the smallest lattice containing the
canonical class in its complexification --- there is a Hodge isometry $T_X\cong T_Y$,
hence $D^b(X)\cong D^b(Y)$ and then $\FM(X)=\FM(Y)$. We therefore get maps
\begin{align*}
 \tilde c &\colon \kp(T,L_{2d})\to\tau\inv(X,h_X), \\
        c &\colon \kp(T,L_{2d})\maps{\tilde c}\tau\inv(X,h_X)\maps{\pi}\sigma\inv(X,h_X) .
\end{align*}
with the fiber $\tau\inv(X,h_X)$ consisting of FM partners of $(X,h_X)$ up to isomorphism.

The map $\tilde c$ is surjective (and hence $c$ is, as well): if
$(Y,h_Y)\in\kf_{2d}$ is an FM partner of $X$, then we first fix a marking
$\lambda_Y\colon H_Y\isom L_{2d}$. By the derived Torelli theorem, there is a Hodge
isometry $g\colon T_X\isom T_Y$. Using $g$ and the markings for $X$ and $Y$, we produce
an embedding
\[ \iota\colon \xymatrix{
   T_X \ar[r]^-\sim_-g & T_Y \ar@{^{(}-}[r] & H_Y \ar[r]^-\sim_-{\lambda_Y} & L_{2d} \ar[r]^-\sim_-{\lambda_X\inv} & H_X
} . \]
This gives a point $\iota\in\kp(T_X,H_X)$ and by construction,
$\tilde c(\iota)=(Y,h_Y)\in\tau\inv(X,h_X)$.

For brevity, we temporarily introduce shorthand notation
\begin{align*}
       \kp^\eq(T_X,H_X) &:= \OO_\hodge(T_X)\times\OO(H_X) \backslash \kp(T_X,H_X), \\
 \tilde\kp^\eq(T_X,H_X) &:= \OO_\hodge(T_X)\times\tilde\OO(H_X) \backslash \kp(T_X,H_X).
\end{align*}
and the goal is to show $\tilde\kp^\eq(T_X,H_X)=\tau^{-1}([X,h_X])$ and
$\kp^\eq(T_X,H_X)=\sigma^{-1}([X,h_X])$.

Now suppose that two embeddings $\iota,\iota'\colon T\embed L_{2d}$ give the same
equivalence class in $\tilde\kp^\eq(T,L_{2d})$. This means that there exist
isometries $g\in\OO(T)$ and  $\tilde g\in\tilde\OO(L_{2d})$ with
 $\iota'\circ g=\tilde g\circ\iota$. Denote the associated polarised $\Kthree$
surfaces by $(Y,h_Y)$ and $(Y',h_{Y'})$; choose markings $\lambda_Y$ and
$\lambda_{Y'}$ as above. Then we obtain a Hodge isometry
\[ \xymatrix{
H_Y \ar[r]^\sim_{\lambda_Y} & L_{2d} \ar[r]^\sim_{\tilde g} &
                           L_{2d} \ar[r]^\sim_{\lambda_{Y'}\inv} & H_{Y'} \\
T_Y \ar@{^{(}-}[u] \ar[r]^\sim & T \ar@{^{(}->}[u]_\iota \ar[r]^\sim_g &
                                 T \ar@{^{(}->}[u]_{\iota'} \ar[r]^\sim & T_{Y'} \ar@{^{(}-}[u]
} \]
and hence $(Y,h_Y)$ and $(Y',h_{Y'})$ define the same point in $\kf_{2d}$.
Thus the map $\tilde c$ factorises over equivalences classes and descends
to a surjective map  $\tilde c\colon \tilde\kp^\eq(T,L_{2d})\to\tau\inv(X,h_X)$.

Analogous reasoning applies if $\iota$ and $\iota'$ are equivalent in
$\kp^\eq(T,L_{2d})$: we get isometries $g\in\OO(T)$ and $\hat g\in\OO(L_{2d})$
with $\iota'\circ g=\hat g\circ\iota$ and use a diagram similar to the one
above. In this case, with the isometry $\hat g$ not necessarily stable,
we can only derive  that the period points coincide in $\hat\kf_{2d}$; hence
$c\colon \kp^\eq(T,L_{2d})\to\sigma\inv(X,h_X)$.

Finally, we show that these maps are injective, as well.
Let $[\iota], [\iota'] \in \tilde\kp^\eq(T,L_{2d})$ be two equivalence classes of
embeddings with $\tilde c([\iota])=\tilde c([\iota'])$. This implies the existence
of a stable isometry in $\OO(L_{2d})$ mapping $\omega\mapsto\omega'$ where $\omega$
and $\omega'$ are given by the construction of the map $\tilde c$ (they correspond
to semi-polarised $\Kthree$ surfaces $(Y,h_Y)$ and $(Y',h_{Y'})$). Using markings
$H_Y\isom L_{2d}$ and $H_{Y'}\isom L_{2d}$, we get an induced Hodge isometry
 $\varphi\colon H_Y\isom H_{Y'}$ with $\varphi(\omega_Y)=\omega_{Y'}$ and
$\varphi(h_Y)=h_{Y'}$. Once more invoking the minimality of transcendental lattices,
we also get a Hodge isometry $\varphi_T\colon T_Y\isom T_{Y'}$. These isometries
combine to
\[ \xymatrix{
L_{2d} \ar[r]^\sim_{\lambda_Y\inv}     & H_Y \ar[r]^\sim_{\varphi} &
H_{Y'} \ar[r]^\sim_{\lambda_{Y'}}  & L_{2d} \\
T \ar@{^{(}->}[u]^\iota \ar[r]^\sim & T_Y \ar@{^{(}-}[u] \ar[r]^\sim_{\varphi_T} &
T_{Y'} \ar@{^{(}-}[u] \ar[r]^\sim   & T \ar@{^{(}->}[u]^{\iota'}
} \]
the outer square of which demonstrates $[\iota]=[\iota']$.

For $[\iota], [\iota'] \in \kp^\eq(T,L_{2d})$ with $c([\iota])=c([\iota'])$,
we argue analogously, only now starting with an isometry of $L_{2d}$ mapping
$\omega\to\omega'$ which is not necessarily stable. Since periods of the $\Kthree$
surfaces get identified up to $\OO(L_{2d})$ in this case, the outer square
gives $[\iota]=[\iota']$ up to $\OO_\hodge(T)$ and $\OO(L_{2d})$.
\qed \end{proof}

\bigskip


From Lemma~\ref{lem:G-eq2} and Proposition \ref{FM-fiber}, we derive the
following statement.

\begin{proposition} \label{prop:fibrecount}
Given an $2d$-polarised $\Kthree$ surface $(X,h_X)$, there are bijections
\begin{align*}
\sigma\inv([X,h_X]) &\bijection \coprod_S \OO_\hodge(T_X)\times\OO(H_X) \backslash \kq(T_X,S,q_{2d}) \\
\tau\inv([X,h_X])   &\bijection \coprod_S \OO_\hodge(T_X)\times\tilde\OO(H_X) \backslash \kq(T_X,S,q_{2d}) .
\end{align*}
where the unions run over isomorphism classes of even lattices $S$ which admit
an overlattice $S\oplus T_X\embed H_X$ such that the induced embedding $S\subseteq H_X$
is primitive. The discriminant is $D_{q_{2d}}=\IZ/2d$ with $q_{2d}(1)=\frac{-1}{2d}$.
\end{proposition}
\begin{proof} \ 
The fibres are obviously partitioned by the orthogonal complements that can
occur (this is in general a bigger choice than
just of an element in the genus).

Once a complement $S$ is chosen, then the set
$\kp^\eq(T_X,S,H_X)$, i.e.\ the set of embeddings of $T_X$ into $H_X$ with complement
isomorphic to $S$, up to Hodge isometries of $T_X$ and isometries of $H_X$,
coincides with $(\OO(H_X),\OO_\hodge(T_X))$-equivalence classes; this
follows from Lemma \ref{lem:G-eq2}. Ana\-lo\-gous reasoning addresses the
$\sigma$-fibres.
\qed \end{proof}


\begin{remark} \label{rem:unpolarised-counting-formula}
Following Remark~\ref{rem:unpolarized} and Lemma~\ref{lem:nikulin151}, we note that the formula for FM partners of
an unpolarised $\Kthree$ surface $X$ from \cite{HLOY2002} can be written as
\begin{align*}
\FM(X) &= \coprod_S \OO_\hodge(T_X)\times\OO(H^2_X) \backslash \kq(T_X,S,0)
\end{align*}
where $S$ now runs through isomorphism classes of lattices admitting an overlattice $S\oplus T_X\embed H^2_X$
such that $S$ is primitive in $H^2_X$. As $H^2_X$ is unimodular, the genus of $S$ is uniquely determined by that
of $T_X$ (see Remark~\ref{rem:unimodular-overlattice}). One candidate for $S$ is $NS(X)$ and we can describe 
$\FM(X)$ as the same union, with $S$ running through the genus $\kg(NS_X)$.
By Lemma~\ref{lem:morrison}, the sets $\kq(T_X,S,0)$ are all mutually bijective.

We will temporarily work with the sets $\kp(T_X,S,H^2_X)$ instead. On each such set, $\OO_\hodge(T_X)$ and
$\OO(H^2_X)$ act in the natural way. Hence, there are not just bijections $\kp(T_X,S,H^2_X)\bijection\kp(T_X,S',H^2_X)$
but also bijections of the quotients by $\OO_\hodge(T_X)\times\OO(H^2_X)$. We thus get
\begin{align*}
 \FM(X) &\bijection \kg(NS_X) \times \big( \OO_\hodge(T_X) \times \OO(H^2_X) \backslash \kp(T_X,NS_X,H^2_X) \big) \\
        &\bijection \kg(NS_X) \times \big( \OO_\hodge(T_X) \times \OO(H^2_X) \backslash \kq(T_X,NS_X,0) \big) .
\end{align*}
However, by Corollary~\ref{cor:unpolarised-case}, this implies
\begin{align*}
 \FM(X) &\bijection \kg(NS_X) \times \big( \OO_\hodge(T_X) \times \OO(NS_X) \backslash \OO(D_{NS_X}) \big) .
\end{align*}

%
Similar formulae in the polarised (hence non-unimodular) case are generally wrong.
\end{remark}




\section{Examples}\label{sec:examples}

\noindent
Proposition \ref{FM-fiber} phrases the problem of classifying polarised $\Kthree$
surfaces up to derived equivalence in lattice terms. Using the results of 
Section~\ref{sec:overlattices} this can be rephrased as Proposition~\ref{prop:fibrecount}
which clearly makes this a finite problem. Given $h_X\in L_{\kd}$ primitive and
$T_X\subseteq H_X=h_X\orth$, or equivalently, $T_X\subset L_{2d}$, one can (in
principle) list all potential subgroups $H$ of the discriminant group. This,
together with the fact that $\Hom(L_{2d},H)$ is finite, makes it possible to test
all potential overlattice groups.

\subsubsection*{Picard rank one}
We consider the special case of Picard rank 1. Here, $h_X\orth=T_X$. Also, any FM
partner of a $2d$-polarised K3 surface is again canonically $2d$-polarised (since the orthogonal
complement of the transcendental lattice is necessarily of the form $\langle -2d \rangle$). Oguiso
showed that the number of non-isomorphic FM partners is $2^{p(d)-1}$ (where $p(d)$
is the number of prime divisors) \cite{oguiso-primes}. This is also half of the
order of $\OO(D_{L_{2d}})$.

Stellari \cite[Theorem 2.2]{stellari-group} shows: the group $\OO(D_{L_{2d}})/\{\pm\id\}$
acts simply transitively on the fibre $\tau\inv(X,h_X)$. In particular, $\sigma$ is
one-to-one on these points.

We look at the situation from the point of view of Proposition~\ref{prop:fibrecount}. In this case 
$T_X=H_X$ and $\OO_\hodge(T_X)\times\OO(H_X) \backslash \kp(T_X,H_X)$ clearly contains only one element,
which says that the fibre of $\sigma$ contains only one element. The situation is different for 
$\tau$. For this we have to analyse the action of the quotient 
$\OO_\hodge(T_X)\times\OO(H_X) /\OO_\hodge(T_X)\times\tilde\OO(H_X) \cong O(D_{L_{2d}})$.  We note
that $- \id$ is contained in both $\OO_\hodge(T_X)$ and $\OO(H_X)$, and the element $(-\id, -\id)$ acts 
trivially on $\kp(T_X,H_X)$. On the other hand, since the Picard number of $X$ is $1$ it follows that every 
element in $\OO_\hodge(T_X)$ extends to an isometry of $H^2(X)$ which maps $h$ to $\pm h$.
Hence the group $O(D_{L_{2d}})/\langle\pm 1\rangle$ acts transitively and freely on the fibre of $\tau$
showing again that the number of FM-partners equals $2^{p(d)-1}$.

\subsubsection*{Large Picard rank}
For Picard ranks of 12 or more, derived equivalence implies isomorphism since any Hodge 
isometry of $T_X$ lifts to an isometry of $H^2_X$, using \cite[1.14.2]{nikulin}:
we have $\ell(NS_X)=\ell(T_X)\leq\rk(T_X)=22-\varrho(X)$ for the minimal number of
generators of $D_{NS_X}$; the lifting is possible if
$2+\ell(NS_X)\leq\rk(NS_X)=\varrho(X)$ --- hence $\ell\leq10$ and $\varrho\geq12$.
Still, there can be many non-isomorphic polarisations on the same surface. See below
for an example where the fibres of $\tau$ and $\sigma$ can become arbitrarily large
in the case of $\varrho=20$ maximal.

\subsubsection*{Positive definite transcendental lattice}
%

We consider the following candidates for transcendental lattices:
$T=\sqmat{2a}{0}{0}{2b}$ with $a>b>0$.
We denote the standard basis vectors for $T$ by $u$ and $v$, so that $u^2=2a$ and $v^2=2b$.
Note that the only isometries of this lattice are given by sending the basis vectors $u,v$ to $\pm u, \pm v$.
In the lattice $L_{2d}=2U\oplus2E_8(-1)\oplus\langle-2d\rangle$, denote by $l$ a generator
of the non-unimodular summand $\langle-2d\rangle$.

In this setting, we are looking for embeddings
 $\iota_1, \iota_2\colon T\embed L_{2d}$
such that $\iota_1(v)$ and $\iota_2(v)$ belong to different $\OO(L_{2d})$-orbits. This
would then immediately imply that the two embeddings cannot be equivalent. In order to 
show this, we appeal to Eichler's criterion.

Let us restrict to the special case $d=b=p^3$ for a prime $p$. 
Recall that the \emph{divisor} of a vector $w$ is the positive generator of the ideal $(w,L_{2d})$.
We want the divisor of the 
vector $v$ to be $p^2$. Setting, for $c\in\IZ$,
\[ v_c := p^2e_2 + p(1+c^2)f_2 + c l , \]
we have $v_c^2=2p^3$ and $\div(v_c)=(p^2,p(1+c^2),2cp^3)$. Choosing $c$ with $1+c^2\equiv0~(p)$,
which enforces $p\equiv1~(4)$, we get $\div(v_c)=p^2$. Now, by Eichler's criterion, the
$\tilde\OO(L_{2p^3})$-orbit of $v_c$ is determined by the length $v_c^2=2p^3$ and the class
$[v_c/\div(v_c)]=[c/p^2]\in D_{L_{2p^3}}$. The latter discriminant group is cyclic of order $2p^3$.
Hence, the number of orbits of vectors with length $2p^3$ and divisor $p^2$ equals the number of
solutions of $1+c^2\equiv0~(p)$ for $c=0,\ldots p^2-1$. The equation $1+c^2\equiv0$ has two
solutions in $\IZ/p$, as $p\equiv1~(4)$, hence $2p$ solutions in $\IZ/p^2$.
(The above computation is a very special case of the obvious adaption of \cite[Prop.\ 2.4]{ghs}
to the lattice $L_{2d}$.)

Together with $u:=e_1+af_1$, we get $2p$ lattices $T_c=\langle u,v_c\rangle$ embedded into
$L_{2p^3}$, and such that these embeddings are pairwise non-equivalent under the action of
$\tilde\OO(L_{2p^3})$. The discriminant group of $L_{2p^3}$ is $D_{L_{2p^3}}\cong\IZ/2p^3$,
hence $\OO(D_{L_{2p^3}})=\{\pm\id\}\cong\IZ/2$. We have to take the action of
 $\OO(L_{2p^3})/\tilde\OO(L_{2p^3})\cong\OO(D_{L_{2p^3}})$ on the set of
$\tilde\OO(L_{2p^3})$-orbits into account. As this is a 2-group, there must at least by $p$
orbits under the action of $\OO(L_{2p^3})$.
We remark that $\OO(D_{L_{2p^3}})$ is a 2-group in greater generality, see \cite[Prop.\ 2.5]{ghs}.

In particular, the number of pairwise non-equivalent embeddings is finite, but unbounded.

\subsubsection*{Unimodular $(T)\orth_{L_{2d}}$}
We use that there are precisely two inequivalent negative definite
unimodular even lattices of rank 16, namely $2E_8$ and $D_{16}^+$
(the latter is an extension of the non-unimodular root lattice $D_{16}$);
see \cite[\S16.4]{conway-sloane}. They become equivalent after adding a
hyperbolic plane: $2E_8\oplus U \cong D_{16}^+\oplus U$,
since unimodular indefinite lattice are determined by rank and signature.
Setting $T:=2U\oplus\langle-2d\rangle$, we get
\[ 2E_8(-1)\oplus T \cong L_{2d} \cong D_{16}^+(-1)\oplus T .\]
Hence, since the orthogonal complements are different, there must at least be
two different embeddings of $T$ into $L_{2d}$. This example allows for
arbitrary polarisations (in contrast to the previous one).

\end{document}